\documentclass[11pt,a4paper,reqno]{amsart}
\pdfoutput=1

\usepackage[T1]{fontenc}
\usepackage[utf8]{inputenc}
\usepackage{lmodern}
\usepackage[english]{babel}
\usepackage{csquotes}
\usepackage[activate]{pdfcprot}

\usepackage{datetime}

\usepackage{amssymb}
\usepackage{amscd}
\usepackage{amsfonts}
\usepackage{mathtools}

\usepackage[pdftex,colorlinks,linkcolor=black,citecolor=black,filecolor=black]{hyperref}

\renewcommand\labelenumi{(\roman{enumi})}
\renewcommand\theenumi\labelenumi

\numberwithin{equation}{section}
\theoremstyle{plain}
\newtheorem{theorem}{Theorem}[section]
\newtheorem{proposition}[theorem]{Proposition}
\newtheorem{lemma}[theorem]{Lemma}
\newtheorem{corollary}[theorem]{Corollary}
\theoremstyle{remark}
\newtheorem{remarkx}[theorem]{Remark}
\newenvironment{remark}
{\pushQED{\qed}\remarkx}
 {\popQED\endremarkx}
\theoremstyle{definition}
\newtheorem{definition}[theorem]{Definition}
\newtheorem{examplex}[theorem]{Example}
\newenvironment{example}
{\pushQED{\qed}\examplex}
 {\popQED\endexamplex}

\newcommand\calS{\mathcal{S}}
\newcommand\calH{\mathcal{H}}
\newcommand\g{\mathfrak{g}}
\DeclareMathOperator{\Lie}{Lie}
\DeclareMathOperator{\Aut}{Aut}
\DeclareMathOperator{\Stab}{Stab}
\DeclareMathOperator{\supp}{supp}
\DeclareMathOperator{\Ad}{Ad}
\newcommand{\Pro}{\mathbb{P}}
\newcommand{\E}{\mathbb{E}}
\DeclareMathOperator{\GL}{GL}
\DeclareMathOperator{\SL}{SL}
\DeclareMathOperator{\PGL}{PGL}
\DeclareMathOperator{\SO}{SO}
\DeclareMathOperator{\Orth}{O}
\DeclareMathOperator{\Gr}{Gr}
\newcommand{\pt}{\bigwedge_{\mathrm{p}}}
\newcommand{\acts}{\boldsymbol{\cdot}}
\newcommand{\leqs}{\leqslant}

\newcommand{\ren}{{\mathrm{r}}}

\newcommand\N{\mathbb{N}}

\newcommand\R{\mathbb{R}}
\newcommand\Z{\mathbb{Z}}

\newcommand{\dd}{\mathop{}\!\mathrm{d}}
\newcommand{\adm}[2]{\ensuremath{(#2{\leftarrow}#1)}}

\DeclarePairedDelimiter\br{(}{)}
\DeclarePairedDelimiter\sqbr{[}{]}
\DeclarePairedDelimiter\abs{\lvert}{\rvert}
\DeclarePairedDelimiter\norm{\lVert}{\rVert}

\providecommand\for{}
\newcommand\SetSymbol[1][]{%
\nonscript\:#1\vert
\allowbreak
\nonscript\:
\mathopen{}}
\DeclarePairedDelimiterX\set[1]{\lbrace}{\rbrace}{%
\renewcommand\for{\SetSymbol[\delimsize]}
#1
}

\newcommand{\euler}{\mathrm{e}}

\renewcommand{\epsilon}{\varepsilon}

\title[Markov random walks \& Diophantine approximation]{Markov random walks on homogeneous spaces and Diophantine approximation on fractals}

\author{Roland Prohaska}
\address{Departement Mathematik, ETH Z\"{u}rich, R\"{a}mistrasse 101, 8092 Z\"{u}rich, Switzerland}
\email{roland.prohaska@math.ethz.ch}
\author{Cagri Sert}
\address{Departement Mathematik, ETH Z\"{u}rich, R\"{a}mistrasse 101, 8092 Z\"{u}rich, Switzerland}
\curraddr{Institut f\"{u}r Mathematik, Universit\"{a}t Z\"{u}rich, Winterthurerstrasse 190, 8057 Z\"{u}rich, Switzerland}
\email{cagri.sert@math.uzh.ch}
\thanks{The second-named author was supported by SNF grants 152819 and 178958.}

\date{\usdate\today}

\begin{document}

\begin{abstract}
In a first part, using the recent measure classification results of Eskin--Lindenstrauss, we give a criterion to ensure a.s.\ equidistribution of empirical measures of an i.i.d.\ random walk on a homogeneous space $G/\Gamma$. 
Employing renewal and joint equidistribution arguments, this result is generalized in the second part to random walks with Markovian dependence. 
Finally, following a strategy of Simmons--Weiss, we apply these results to Diophantine approximation problems on fractals and show that almost every point with respect to Hausdorff measure on a graph directed self-similar set is of generic type, so in particular, well approximable. 
\end{abstract}
\maketitle
\section{Introduction}\label{sec:intro}
For the introduction, let $G$ be a connected simple real Lie group and $\Gamma$ a lattice in $G$. 
Let $(Y_n)_n$ be a $G$-valued stochastic process and $x_0\in X=G/\Gamma$. 
These data define a \emph{random walk} on $X$: 
Starting at $x_0$, one consecutively applies the random group elements $Y_1$, $Y_2$, etc. 
One of the main objectives of this paper is to identify a set of conditions on the increment process $(Y_n)_n$ ensuring that the random walk trajectory
\begin{align*}
(Y_n\dotsm Y_1x_0)_n
\end{align*}
almost surely equidistributes towards the normalized Haar measure $m_X$ on $X$ for every $x_0\in X$, meaning convergence
\begin{align*}
\frac{1}{n}\sum_{k=0}^{n-1}\delta_{Y_k\dotsm Y_1x_0}\longrightarrow m_X
\end{align*}
in the weak* topology as $n\to\infty$. 
\subsection{I.I.D.\ Random Walks}\label{subsec:iid_intro}
We first consider the classical case where the increments $Y_n$ are independent and identically distributed (i.i.d.). 
Two different types of assumptions on the common distribution $\mu$ have previously been used to establish equidistribution results in this context. 
The first one concerns the algebraic structure of the support of $\mu$ and was studied by Benoist--Quint. 
A special case of their results in~\cite{BQ3} is the following.
\begin{theorem}[Benoist--Quint~\cite{BQ3}]\label{thm:BQ_intro}
Let $\mu$ be a compactly supported probability measure on $G$ and suppose that the closed subgroup $G_\calS$ generated by $\calS=\supp(\mu)$ has the property that $\Ad(G_\calS)$ is Zariski dense in $\Ad(G)$. 
Let $(Y_n)_n$ be a sequence of i.i.d.\ random variables with distribution $\mu$. 
Then for every $x_0\in X$ with infinite $G_\calS$-orbit, the random walk trajectory $(Y_n\dotsm Y_1x_0)_n$ almost surely equidistributes towards $m_X$.
\end{theorem}
The second set of assumptions involves the dynamics of the linearized random walk on the Lie algebra $\g$ of $G$. 
In~\cite{SW}, Simmons--Weiss impose the following requirements on the adjoint action of $G_\calS=\overline{\langle\supp(\mu)\rangle}$ on $\g$, phrased in terms of Oseledets subspaces (see Theorem~\ref{thm:oseledets} for their definition): 
\begin{enumerate}
\item[(I)] For every $1\le k\le \dim(G)-1$ there exists a proper non-trivial $G_\calS$-invariant subspace $W_k\subset \g^{\wedge k}$ such that, almost surely, $W_k$ trivially intersects the Oseledets subspace  $V^{\leqs 0}$ of subexponential expansion, and $W\coloneqq W_1$ is complementary to the Oseledets subspace $V^{<\max}$ of non-maximal expansion.
\item[(II)] The adjoint action of $G_\calS$ on $W$ is by similarities and satisfies
\begin{align*}
\int_G\log\norm{\Ad(g)|_W}\dd\mu(g)>0.
\end{align*}
\item[(III)] For $1\le k\le \dim(G)-1$, any non-trivial subspace $L\subset \g^{\wedge k}$ with finite orbit under $G_\calS$ intersects $W_k$ non-trivially.
\end{enumerate}
A model example to have in mind is the action of the Borel subgroup (endowed with a suitable measure) on the Lie algebra of the upper unipotent subgroup in~$\SL_2(\R)$. 

Modifying the arguments in~\cite{BQ1}, Simmons--Weiss prove the following theorem. 
For the statement, recall that a subgroup $H$ of $G$ is said to be \emph{virtually contained} in a subgroup $L$ of $G$ if $H\cap L$ has finite index in $H$. 
\begin{theorem}[Simmons--Weiss~\cite{SW}]\label{thm:SW_main}
Let $\mu$ be a compactly supported probability measure on $G$ such that the closed subgroup $G_\calS$ generated by $\calS=\supp(\mu)$ is not virtually contained in any conjugate of $\Gamma$ and suppose that conditions \textup{(I)--(III)} are satisfied. 
Let $(Y_n)_n$ be a sequence of i.i.d.\ random variables with distribution $\mu$. 
Then for every $x_0\in X$, the random walk trajectory $(Y_n\dotsm Y_1x_0)_n$ almost surely equidistributes towards $m_X$.
\end{theorem}
Note that the virtual containment condition in the above theorem is equivalent to saying that there do not exist finite $G_\calS$-orbits in $X$.

Simmons--Weiss' conditions (I) \& (III) and Benoist--Quint's assumption of Zariski density of $\Ad(G_\calS)$ in the simple group $\Ad(G)$ are mutually exclusive. 
However, what the two settings have in common is that both imply what we shall call \emph{uniform expansion on Grassmannians} (see~\S\ref{subsec:grass}): 
For $1\le k\le \dim(G)-1$ and every non-zero pure wedge product $v=v_1\wedge \dots \wedge v_k$ in $\g^{\wedge k}$, almost surely, 
\begin{align}\label{exp}
\liminf_{n\to\infty}\tfrac{1}{n}\log\norm{\Ad^{\wedge k}(Y_n\dotsm Y_1)v}>0.
\end{align}
Elaborating on the recent measure classification results of Eskin--Lindenstrauss in~\cite{EL}, we show that this expansion property is sufficient to guarantee almost sure equidistribution. 
Moreover, their work allows replacing compact support of $\mu$ by \emph{finite exponential moments} in $\g$, meaning that $N(g)=\max(\norm{\Ad(g)},\norm{\Ad(g)^{-1}})$ satisfies
\begin{align*}
\int_GN(g)^\delta\dd\mu(g)<\infty
\end{align*}
for some $\delta>0$. 
We prove the following. 
\begin{theorem}\label{thm:intro1}
Let $\mu$ be a probability measure on $G$ with finite exponential moments in $\g$ such that the closed subgroup $G_\calS$ generated by $\calS=\supp(\mu)$ is not virtually contained in any conjugate of $\Gamma$. 
Suppose that the i.i.d.\ process $(Y_n)_n$ with common law $\mu$ is uniformly expanding on Grassmannians. 
Then for every $x_0\in X$, the random walk trajectory $(Y_n\dotsm Y_1x_0)_n$ almost surely equidistributes towards $m_X$. 
\end{theorem}
We will establish this result in the slightly more general form of Theorem~\ref{thm:exp_equidist}. 
The proof breaks down into the two usual steps: 
\begin{itemize}
\item Classification of stationary measures (Theorem~\ref{thm:exp_stat}): 
This step essentially follows from the work of Eskin--Lindenstrauss~\cite{EL}, but an additional argument is required to upgrade their classification to the statement we need, namely that the only non-atomic $\mu$-stationary probability measure on $X$ is the Haar measure $m_X$.
\item Ruling out escape of mass (Proposition~\ref{prop:non_escape}): 
Here the key ingredient is Eskin--Margulis' work on non-divergence~\cite{EM}, which we exploit along the same lines as in the proof of \cite[Theorem~2.1]{SW}. 
\end{itemize}

As one of the consequences of Theorem~\ref{thm:intro1}, we show that assumptions (I)--(III) above can be relaxed to the following two conditions: 
\begin{enumerate}\label{relaxed_conditions}
\item[(I')] For every $1\le k\le \dim(G)-1$ there exists a proper non-trivial $G_\calS$-invariant subspace $W_k\subset \g^{\wedge k}$ such that, almost surely, $W_k$ trivially intersects the Oseledets subspace  $V^{\leqs 0}$ of subexponential expansion.
\item[(III')] For $1\le k\le \dim(G)-1$, any non-trivial $G_\calS$-invariant subspace $L$ of $\g^{\wedge k}$ intersects $W_k$ non-trivially. 
\end{enumerate}
A simple example in which (I') and (III') hold whereas (I)--(III) fail is given by $G=\SL_3(\R)$, $\Gamma=\SL_3(\Z)$ and $\mu=\tfrac13(\delta_{g_1}+\delta_{g_2}+\delta_{g_3})$ for the matrices
\label{intro_example}
\begin{align*}
g_1=\begin{pmatrix}3&&\\&2&\\&&1/6\end{pmatrix},\,g_2=\begin{pmatrix}3&&1\\&2&\\&&1/6\end{pmatrix}\text{ and }g_3=\begin{pmatrix}3&&\\&2&1\\&&1/6\end{pmatrix}.
\end{align*}
We postpone the justification to~\S\ref{subsec:grass}. 
\subsection{Markov Random Walks}\label{subsec:markov_intro}
The properties of random products $Y_n\dotsm Y_1$ of elements of $G$ are much less understood when the increments $Y_n$ do not form an i.i.d.\ process. 
The problem of equidistribution on homogeneous spaces, for instance, has not been studied beyond the case of i.i.d.\ random walks. 
In this article, we investigate this problem for Markovian increment processes and, as our main result, obtain equidistribution results analogous to the i.i.d.\ case. 
\begin{theorem}\label{thm:intro2}
Let $(Y_n)_n$ be an irreducible Markov chain on a finite set $\calS\subset G$ that is uniformly expanding on Grassmannians in the sense of \eqref{exp} and such that for every $x\in X$ the random orbit $\set{Y_n\dotsm Y_1x\for n\in\N}$ is almost surely infinite. 
Then for every $x_0\in X$, the random walk trajectory $(Y_n\dotsm Y_1x_0)_n$ almost surely equidistributes towards $m_X$. 
\end{theorem}
Note that when $\mu$ is finitely supported, $G_\calS$ denotes the closed subgroup of $G$ generated by $\calS=\supp(\mu)$, and the $Y_n$ are i.i.d.\ with distribution $\mu$, the random orbit $\set{Y_n\dotsm Y_1x\for n\in\N}$ is almost surely infinite if and only if the orbit $G_\calS x$ is infinite. 
%Indeed, this can be seen by studying the discrete Markov chain on $G_\calS x$ given by the random walk. 
Hence, the condition on almost surely infinite orbits in Theorem~\ref{thm:intro2} is a natural analogue of the virtual containment condition in Theorem~\ref{thm:intro1}. 
%Theorem~\ref{thm:intro2} will follow from Theorem~\ref{thm:exp_markov_equidist}. 

The proof of Theorem~\ref{thm:intro2} relies on Theorem~\ref{thm:intro1} and a renewal argument. 
Indeed, our strategy of proof is to apply Theorem~\ref{thm:intro1} to the blocks $Z_n=Y_{\tau_g^{n+1}-1}\dotsm Y_{\tau_g^n}$ between consecutive hitting times $\tau_g^n$ and $\tau_g^{n+1}$ of a fixed state $g\in G$, which are i.i.d.\ by the Markov property of $(Y_n)_n$, and then deal with the excursions between such hitting times. 
By the strong recurrence properties of finite-state Markov chains, these excursions are rather short most of the time, so that their contribution can be precisely controlled thanks to a joint equidistribution phenomenon (see~\S\ref{subsec:markov_equidist_bootstrap}). 
Note that for this approach to work, it is crucial that Theorem~\ref{thm:intro1} does not require $\mu$ to have compact support as in Theorems \ref{thm:BQ_intro} and \ref{thm:SW_main}. 

A concrete corollary of the previous result is the following Markovian version of \cite[Theorem~1.1]{SW}. 
\begin{corollary}\label{cor:intro4}
Let $G=\SL_{d+1}(\R)$, $\Gamma=\SL_{d+1}(\Z)$, and $X=G/\Gamma$. For $0\le i\le r$ let $c_i>1$ be real numbers, $y_i\in \R^d$ vectors such that $y_0=0$ and $y_1,\dots,y_r$ span $\R^d$, $O_i\in \SO_d(\R)$, and set
\begin{align*}
g_i=\begin{pmatrix}c_iO_i&y_i\\0&c_i^{-d}\end{pmatrix}\in G.
\end{align*}
Then for any irreducible Markov chain $(Y_n)_n$ on $\calS=\set{g_0,\dots,g_r}\subset G$ with one universally accessible state \textup{(}i.e.\ a state that can be reached in a single step from everywhere with positive probability\textup{)} and any starting point $x_0\in X$, the random walk trajectory $(Y_n\dotsm Y_1x_0)_n$ almost surely equidistributes towards $m_X$.
\end{corollary}
%We will deduce this corollary in~\S\ref{subsec:upper_block_form}. 
We remark that, in this corollary, the assumption of having a universally accessible state plays the role of an aperiodicity condition, which allows deducing the dynamical property of uniform expansion on Grassmannians from the algebraic structure of the set $\calS$. 
Without such a condition, excursions from a fixed state might fail to witness this structure in full, and degenerate behavior may occur. 
\subsubsection{Beyond Markov}
An advantage of an expansion condition such as \eqref{exp} over one involving the measure $\mu$ is that it puts the i.i.d.\ case on equal footing with arbitrary increment processes. 
Consequently, the formulation of Theorem~\ref{thm:intro2} suggests the natural question of equidistribution for more general, say ergodic and stationary, increment processes $(Y_n)_n$ on $G$. 
For example, one might expect Theorem~\ref{thm:intro2} to hold true when, instead of being a Markov process, the distribution of $(Y_n)_n$ is a Gibbs measure of some H\"older continuous potential on $\calS^\N$. 
While our approach in this article can handle locally constant potentials (corresponding to generalized Markov measures), the general question remains open.
\subsection{Applications to Diophantine Approximation on Fractals}\label{subsec:fractal_intro}
Recall that by a classical theorem of Dirichlet, for any $v\in\R^d$, there exist infinitely many pairs $(p,q)\in\Z^d\times\N$ such that $\norm{qv-p}_\infty\le q^{-1/d}$. 
If for some constant $c\in(0,1)$ there are only finitely many solutions $(p,q)\in\Z^d\times \N$ to the stronger inequality $\norm{qv-p}_\infty\le cq^{-1/d}$, then $v$ is said to be \emph{badly approximable}, and \emph{well approximable} otherwise. 
The set of badly approximable points in $\R^d$ is of zero Lebesgue measure (but of full Hausdorff dimension). 

In the study of Diophantine approximation on fractals, one is in particular interested in Diophantine properties of typical points of a fractal in $\R^d$ with respect to natural measures on that fractal; most prominently, Hausdorff measure. 
In the absence of algebraic obstructions, it is generally expected that these properties are the same as for Lebesgue-typical points of the ambient space $\R^d$. 
However, for badly approximable points this analogy remained poorly understood after the initial results of Einsiedler--Fishman--Shapira~\cite{EFS} that concerned a somewhat restricted class of fractals. 

The recent breakthrough of Simmons--Weiss~\cite{SW} contributed considerably to this problem, showing in particular that for an irreducible iterated function system (IFS) $\Phi=\set{\phi^{(1)},\dots,\phi^{(k)}}$ of contracting similarities of $\R^d$ and any Bernoulli measure $\beta$ on $\Phi^\N$ of full support, almost every point of the associated self-similar fractal is of generic type, where \enquote{almost every} is understood with respect to the pushforward of $\beta$ by the \emph{natural projection}
\begin{align*}
\Pi\colon\Phi^\N\to\R^d,\,(\phi_m)_m\mapsto\lim_{n\to\infty}\phi_0\dotsm\phi_{n-1}(x),
\end{align*}
where $x\in\R^d$ is arbitrary. 
Thanks to a classical result of Hutchinson~\cite{Hut}, this implies the same conclusion with respect to Hausdorff measure whenever the IFS satisfies the open set condition. 
Here, a point being of \enquote{generic type} intuitively means that, from a Diophantine approximation perspective, it behaves like a Lebesgue-typical point in $\R^d$. 
In particular, such points are well approximable. 
When $d=1$, this property also implies that the blocks of the continued fraction expansion are distributed according to Gauss measure. 
For the precise definition see~\S\ref{subsec:dioph}. 

In our main applications below, following the strategy in~\cite{SW} and making use of our Markovian equidistribution results, we extend the aforementioned results of~\cite{SW} in two directions. 

The first one concerns measures that are not necessarily Bernoulli. 
For the statement, recall that an IFS $\Phi=\set{\phi^{(1)},\dots,\phi^{(k)}}$ of contracting similarities of $\R^d$ is said to be \emph{irreducible} if there does not exist a proper affine subspace $V\subset\R^d$ that is preserved by all $\phi^{(i)}$, and that its \emph{attractor} is the unique non-empty compact set $K\subset\R^d$ with $K=\bigcup_{i=1}^k\ \phi^{(i)}(K)$. 
Equivalently, the attractor $K$ can be written as the image of $\Phi^\N$ under the natural projection $\Pi$ defined above. 
\begin{theorem}\label{thm:intro3}
Let $\Phi$ be an irreducible IFS of contracting similarities of $\R^d$, $K$ the associated attractor, and $\Pi\colon\Phi^\N\to\R^d$ the natural projection. 
Then for any Markov measure $\Pro$ on $\Phi^\N$ of full support, $\Pi_*\Pro$-a.e.\ point on $K$ is of generic type, so in particular, well approximable. 
\end{theorem}
Under a strong separation condition, the statement about well approximable points in the above theorem also follows from Simmons--Weiss' \cite[Theorem~8.4]{SW} on doubling measures. 
However, in general the measures in our theorem are not doubling on the attractor $K$, even under the open set condition; see~\cite{Yung}.

Secondly, we consider more general, no longer strictly self-similar fractals $K$. 
Given an IFS $\Phi$ of contracting similarities, these fractals are obtained as images under the natural projection $\Pi$ of \emph{sofic subshifts} of the shift space $\Phi^\N$, which are by definition continuous factors of subshifts of finite type~\cite{Wei}. 
Accordingly, we call the associated fractals \emph{sofic similarity fractals}. 

In the literature, the iterated function systems appearing in the construction of such fractals are known as \emph{graph directed IFS}, since a sofic shift can always be realized as image of the edge shift of a directed graph under a one-block factor map (see e.g.\ \cite{LM_book}). 
Each edge in the graph has as label one of the similarities in $\Phi$ and the possible paths in the graph determine the sequences appearing in the sofic shift. 
Since its introduction by Mauldin--Williams~\cite{MW}, this viewpoint has proved to be a fruitful approach and has been studied by many authors, among others Edgar--Mauldin~\cite{EdM}, Olsen~\cite{Ol}, Wang~\cite{Wa}, and Mauldin--Urbanski in their monograph~\cite{MU_book}. 
For an accessible introduction we refer to Edgar's book~\cite{Ed_book}.

The advantage of this setup over the point of view of an abstract sofic shift is that classical properties of an IFS like the open set condition or irreducibility can be expressed in a more lucid and conceptual way. 
With these notions, which will be defined in~\S\ref{subsec:gd_IFS}, we have the following result. 
\begin{theorem}\label{thm:intro4}
Let $K\subset \R^d$ be a sofic similarity fractal constructed by a graph directed IFS of contracting similarities that is irreducible and satisfies the open set condition. 
Let $s\ge 0$ denote the Hausdorff dimension of $K$. 
Then almost every point on $K$ with respect to $s$-dimensional Hausdorff measure is of generic type, so in particular, well approximable. 
\end{theorem}
\subsection*{Terminology, Notation, Conventions}
In the whole article, $G$ is a real Lie group with Lie algebra $\g$ and $X$ is a locally compact second countable metrizable space on which $G$ acts continuously. 
Frequently, $X$ will be the homogeneous space $G/\Gamma$ for a discrete subgroup $\Gamma$ of $G$. 
In case $\Gamma$ is a lattice, we write $m_X$ for the unique $G$-invariant Borel probability measure on $X$, which we simply refer to as the \emph{Haar measure} on $X$. 
When $G$ is endowed with a Borel probability measure $\mu$, we write $\mu^{*n}$ for the $n^{\text{th}}$ convolution power of $\mu$. 
Throughout, we fix a scalar product $\langle\cdot,\cdot\rangle$ on $\g$, which induces scalar products on the exterior powers of $\g$ given by
\begin{align*}
\langle v_1\wedge\dots\wedge v_k,w_1\wedge \dots\wedge w_k\rangle=\det(\langle v_i,w_j\rangle)_{1\le i,j\le k}
\end{align*}
for pure wedge products and extended bilinearly to all of $\g^{\wedge k}$. 
The induced norms are all denoted by $\norm{\cdot}$. 
This should cause no confusion.

In the sequel, we shall not take the point of view of stochastic processes as in the introduction, but rather work with the canonical coordinate process on the product space $B=G^\N$, governed by some probability measure on it. 
In the i.i.d.\ case, that measure is the product measure $\beta=\mu^{\otimes \N}$. 
In the Markovian case it will in fact be advantageous to not work directly in $G$, but with an abstract set $E$ that is mapped to $G$ via some \emph{coding map} $E\ni e\mapsto g_e\in G$. 
The measures governing our processes will then be Markov measures on $\Omega=E^\N$. 
The shift map on $\Omega$ will be denoted by $T$. 
We shall also need to deal with the semigroup $E^*$ of finite words over $E$. 
The length of a word $w$ is denoted by $\ell(w)$. 
The coding map $e\mapsto g_e$ naturally extends to a homomorphism $E^*\to G$ given by $g_w=g_{e_{n-1}}\dotsm g_{e_0}\in G$ for a word $w=e_{n-1}\ldots e_0\in E^*$. 
For $\omega=(\omega_m)_m\in\Omega$ and $n\in\N$ we shall write $\omega|_n$ for the finite word $\omega_{n-1}\ldots\omega_0\in E^*$. 

An important special case of the above is the choice $E=G$ with the identity map as coding map. 
In this case, we have $g_{b|_n}=b_{n-1}\dotsm b_0$ for $b=(b_m)_m\in B$ and $n\in\N$, and $T$ is the shift map on $B$. 

For a finite-dimensional real vector space $V$, we write $\Pro(V)$ for the projective space associated to $V$. 
Given a (continuous) representation $\rho$ of $G$ on $V$, we set $N(g)=\max(\norm{\rho(g)},\norm{\rho(g)^{-1}})$ for $g\in G$, the norm being the operator norm coming from some fixed norm on $V$. 
The probability measure $\mu$ on $G$ is said to have a \emph{finite first moment} in $(V,\rho)$ if
\begin{align*}
\int_G\log N(g)\dd\mu(g)<\infty,
\end{align*}
and to have \emph{finite exponential moments} in $(V,\rho)$ if
\begin{align*}
\int_G N(g)^\delta\dd\mu(g)<\infty
\end{align*}
for sufficiently small $\delta>0$. 
When $(V,\rho)=(\g,\Ad)$, we shall omit the representation from the notation and simply speak of finite first or exponential moments in $\g$. 

We say that a sequence $(y_n)_n$ in a Polish space $Y$ \emph{equidistributes} towards a probability measure $\eta$ on $Y$ if 
\begin{align*}
\lim_{n\to\infty}\frac1n\sum_{k=0}^{n-1}f(y_k)=\int_Y f\dd\eta
\end{align*}
for every bounded continuous function $f$ on $Y$. 
Equidistribution of sequences in the (locally compact) space $X$ can be expressed in terms of weak* convergence as follows: 
By definition, a sequence $(\nu_n)_n$ of probability measures on $X$ converges to a finite measure $\nu$ on $X$ in the weak* topology if
\begin{align}\label{weak*_conv}
\lim_{n\to\infty}\int_Xf\dd\nu_n=\int_Xf\dd\nu
\end{align}
for every compactly supported continuous function $f$ on $X$. 
The limit measure $\nu$ always satisfies $\nu(X)\le 1$. 
When $\nu$ is a probability measure, weak* convergence of $(\nu_n)_n$ to $\nu$ implies that \eqref{weak*_conv} holds for all bounded continuous functions. 
Consequently, a sequence $(x_n)_n$ in $X$ equidistributes towards a probability measure $\nu$ on $X$ if and only if the \emph{empirical measures} $\frac1n\sum_{k=0}^{n-1}\delta_{x_k}$ converge to $\nu$ in the weak* topology as $n\to\infty$. 

Finally, a probability measure $\nu$ on $X$ is said to be \emph{$\mu$-stationary} if $\mu*\nu=\nu$, where the convolution $\mu*\nu$ is defined by $\mu*\nu=\int_Gg_*\nu\dd\mu(g)$, or in other words by $\int_X f \dd(\mu*\nu)=\int_X\int_G f(gx)\dd\mu(g)\dd\nu(x)$ for every bounded measurable function $f$ on $X$. 
A $\mu$-stationary probability measure $\nu$ is called \emph{$\mu$-ergodic} if it is extremal in the weak*-closed convex set of $\mu$-stationary probability measures on $X$. 
\subsection*{Acknowledgments}
The authors would like to express their gratitude towards Manfred Einsiedler for helpful discussions, valuable insights and his encouragement to pursue the topic at hand, towards Alex Eskin and Elon Lindenstrauss for making available preprint versions of the paper~\cite{EL} and indulging the authors' questions about it, and towards Jean-Fran\c{c}ois Quint for useful discussions. 
Thanks also go to the anonymous referee for the remarks on the manuscript, and in particular for pointing out a simplification of part of the proof of Proposition~\ref{prop:expansion_equivalence}. 
\section{I.I.D.\ Random Walks}\label{sec:iid_walks}
In this section, we investigate i.i.d.\ random products satisfying certain expansion conditions. 
After recalling some classical facts about random matrix products in~\S\ref{subsec:prelim}, these conditions are defined and studied in~\S\ref{subsec:exp_proj} and~\S\ref{subsec:grass}. 
Afterwards, we state and prove measure classification and equidistribution results in~\S\ref{subsec:iid_results}. 
The main result is Theorem~\ref{thm:exp_equidist}, which implies Theorem~\ref{thm:intro1}. 
The employed arguments rely on Eskin--Lindenstrauss' results in~\cite{EL}. 
%As a by-product, our approach yields alternative proofs of key steps in the proof of Simmons--Weiss' first main result in~\cite{SW} (reproduced here as Theorem~\ref{thm:SW_main}); see Remarks~\ref{rmk:SW_simplification1}~and~\ref{rmk:SW_simplification2}. 

Throughout this section, $\mu$ is a probability measure on $G$ with support $\calS$ and $G_\calS$ denotes the closed subgroup of $G$ generated by $\calS$. 
\subsection{Preliminaries on Random Matrix Products}\label{subsec:prelim}
We start by recalling two fundamental results about exponential growth rates for random matrix products. 
Let $G=\GL_d(\R)$ and assume $\mu$ has a finite first moment. 

The first result is Oseledets' multiplicative ergodic theorem. 
It makes a statement about the \emph{Lyapunov exponents} $\lambda_1(\mu)\ge \dots \ge \lambda_d(\mu)$ of $\mu$, which are the real numbers defined by 
\begin{align*}
\lambda_1(\mu)+\dots+\lambda_i(\mu)\coloneqq\lim_{n\to\infty} \tfrac{1}{n} \E\sqbr*{\log\norm{(g_{b|_n})^{\wedge i}}}\overset{\beta\text{-a.s.}}{=}\lim_{n\to\infty}\tfrac{1}{n}\log\norm{(g_{b|_n})^{\wedge i}}
\end{align*}
for $1 \le i \le d$, where the second equality follows from Kingman's subadditive ergodic theorem and ergodicity of the underlying Bernoulli shift. 
\begin{theorem}[Oseledets~\cite{Os}]\label{thm:oseledets}
Given $\mu$ as above, there exists a shift-invariant measurable subset $B'\subset B$ of $\beta$--full measure such that for every $b\in B'$ 
\begin{enumerate}
\item $\bigl((g_{b|_n})^*(g_{b|_n})\bigr)^{1/2n}$ converges to an invertible symmetric matrix $L_b$, 
\item the eigenvalues of $L_b$ are $\euler^{\lambda_{i_1}(\mu)},\dots,\euler^{\lambda_{i_s}(\mu)}$, where $1\le i_1<\dots<i_s\le d$ are indices chosen such that $\lambda_{i_1}(\mu)>\dots>\lambda_{i_s}(\mu)$ are the distinct Lyapunov exponents of $\mu$, 
\item if $U^{1}_b,\dots,U^{s}_b$ denote the corresponding eigenspaces of $L_b$, the \emph{Oseledets subspaces} $V_b^j\coloneqq U^j_b\oplus\dots\oplus U^s_b$ have the property that 
\begin{align*}
\lim_{n\to\infty}\tfrac{1}{n}\log\norm{g_{b|_n} v}=\lambda_{i_j}(\mu)
\end{align*}
whenever $v\in V_b^{j}\setminus V_b^{j+1}$, and satisfy the equivariance $V^{j}_b=b_1^{-1}V^j_{Tb}$. 
\end{enumerate}
The space $V_b^{<\max}\coloneqq V_b^2$ is called the \emph{Oseledets subspace of non-maximal expansion}. 
The largest Oseledets subspace with non-positive exponent is denoted by $V_b^{\leqs 0}$ and is called the \emph{Oseledets subspace of subexponential expansion} \textup{(}set $V_b^{\leqs 0}=\set{0}$ if there are only positive exponents\textup{)}. 
\end{theorem}
We refer to Ruelle~\cite[\S1]{Ru} for an exposition. 

In contrast to the random nature of Oseledets subspaces, the second result we wish to review describes exponential growth rates along a deterministic filtration. 
\begin{theorem}[Furstenberg--Kifer~\cite{FK}, Hennion~\cite{Hen}]\label{thm:furstenberg_kifer}
Let $\mu$ be as above. 
Then there exists a partial flag $\R^d= F_1\supset F_2\supset \dots\supset F_k\supset F_{k+1}=\set{0}$ of $G_\calS$-invariant subspaces and a collection of real numbers $\lambda_1(\mu)=\beta_1(\mu)>\dots>\beta_k(\mu)$ such that for every $v \in F_i\setminus F_{i+1}$ we have $\beta$-a.s.\ 
\begin{align*}
\lim_{n\to\infty}\tfrac{1}{n}\log \norm{g_{b|_n}v}=\beta_i(\mu).
\end{align*}
Moreover, the $\beta_i(\mu)$ are the values of 
\begin{align*}
\alpha(\nu)\coloneqq\int_{\Pro(\R^d)}\int_G\log\frac{\norm{gv}}{\norm{v}}\dd\mu(g)\dd\nu(\R v)
\end{align*}
that occur when $\nu$ ranges over $\mu$-ergodic $\mu$-stationary probability measures on $\Pro(\R^d)$. 
If $\alpha(\nu)=\beta_i(\mu)$ for such a measure $\nu$, then $v\in F_i\setminus F_{i+1}$ for $\nu$-a.e.\ $\R v\in\Pro(\R^d)$. 
\end{theorem}
When applying the above theorem, we will frequently use the notation $F^{\leqs 0}$ for the maximal subspace $F_i$ with exponent $\beta_i(\mu) \le 0$. 
\subsection{Expansion on Projective Space}\label{subsec:exp_proj}
When all exponents $\beta_i$ in Theorem~\ref{thm:furstenberg_kifer} are positive, all non-zero vectors are expanded by the random matrix product at a uniform exponential rate. 
\begin{definition}[Uniform expansion]
Let $\mu$ be a probability measure on $\GL_d(\R)$ and $P$ a closed $G_\calS$-invariant subset of $\Pro(\R^d)$. 
Then $\mu$ is said to be \emph{uniformly expanding on $P$} if for every $\R v \in P$, for $\beta$-a.e.\ $b\in B$ we have 
\begin{align*}
\liminf_{n\to\infty} \tfrac{1}{n}\log\norm{g_{b|_n}v}>0.
\end{align*}
\end{definition}

In the literature, the idea of uniform expansion has been formalized in different ways, with some relationships established between them (see e.g.\ \cite[Lemma~1.5]{EL}, \cite[\S3]{SW}). 
In the following proposition, we prove the equivalence of the definition we are working with to some of its common variants. 
\begin{proposition}\label{prop:expansion_equivalence}
Let $\mu$ be a probability measure on $\GL_d(\R)$ with finite first moment and let $P$ be a closed $G_\calS$-invariant subset of $\Pro(\R^d)$. 
The following properties are equivalent to uniform expansion of $\mu$ on $P$: 
\begin{enumerate}
\item There exists $N\in\N$ and a constant $C_1>0$ such that for every $\R v \in P$ and every $n\ge N$ we have 
\begin{align*}
\frac1n\int_G\log\frac{\norm{gv}}{\norm{v}}\dd\mu^{*n}(g)\ge C_1>0.
\end{align*}
\item There exists $N\in\N$ and a constant $C_2>0$ such that for every $\R v \in P$ we have 
\begin{align*}
\int_G\log\frac{\norm{gv}}{\norm{v}}\dd\mu^{*N}(g)\ge C_2>0.
\end{align*}
\item For every $\R v \in P$, for $\beta$-a.e.\ $b\in B$ we have 
\begin{align*}
\lim_{n\to\infty}\tfrac{1}{n}\log\norm{g_{b|_n}v}>0.
\end{align*}
\end{enumerate}
\end{proposition}
%\begin{remark}\label{rmk:SW_simplification1}
%For $P=\Pro(\R^d)$, the implication (iii)$\implies$(i) in the above proposition establishes one of the steps in Simmons--Weiss' proof of Theorem~\ref{thm:SW_main} in~\cite{SW}, namely the deduction of \cite[Proposition~3.3]{SW} from \cite[Proposition~3.7]{SW}. 
%\end{remark}
\begin{proof}
We apply Theorem~\ref{thm:furstenberg_kifer}. 
One of its consequences is that the limit in (iii) exists $\beta$-a.s.\ for every $\R v\in P$. 
In particular, we see that (iii) is equivalent to uniform expansion of $\mu$ on $P$. 
Of the remaining implications, only (ii)$\implies$(iii) and (iii)$\implies$(i) require a proof. 

(ii)$\implies$(iii): 
Since the limit in (iii) exists, we may pass to a subsequence of indices and assume $N=1$. 
The set $P\cap\Pro(F^{\leqs 0})$ is a closed $G_\calS$-invariant subset of $\Pro(\R^d)$. 
Assume it is non-empty. 
Then it supports a $\mu$-ergodic $\mu$-stationary probability measure $\nu$ and Theorem~\ref{thm:furstenberg_kifer} implies that $\alpha(\nu)$ occurs as exponential growth rate on $F^{\leqs 0}$. 
However, due to (ii) we have
\begin{align*}
\alpha(\nu)=\int_P\int_G\log\frac{\norm{gv}}{\norm{v}}\dd\mu(g)\dd\nu(\R v)\ge C_2>0,
\end{align*}
a contradiction. 
Hence, $P\cap\Pro(F^{\leqs 0})$ must be empty, which is equivalent to (iii). 

(iii)$\implies$(i): 
We argue by contradiction. 
If (i) does not hold, then there exists a sequence $(\R v_j)_j$ in $P$ and a sequence of integers $(n_j)_j$ with $n_j\to\infty$ such that 
\begin{align}\label{no_expansion}
\limsup_{j\to\infty}\frac{1}{n_j}\int_G\log\frac{\norm{gv_j}}{\norm{v_j}}\dd\mu^{*n_j}(g)\le 0.
\end{align}
Passing to a subsequence if necessary, we may assume that 
\begin{align*}
\frac{1}{n_j}\sum_{k=0}^{n_j-1}\mu^{*k}*\delta_{\R v_j}\longrightarrow \tilde{\nu}
\end{align*}
as $j\to\infty$ in the weak* topology for some limit probability measure $\tilde{\nu}$ on $\Pro(\R^d)$ with support in $P$. 
Note that $\tilde{\nu}$ necessarily is $\mu$-stationary. 
Using the additive cocycle property of $(g,\R v)\mapsto \log\frac{\norm{g v}}{\norm{v}}$ together with \eqref{no_expansion} it follows that 
\begin{align*}
\int_{\Pro(\R^d)}\int_G\log\frac{\norm{gv}}{\norm{v}}\dd\mu(g)\dd\tilde{\nu}(\R v)&=\lim_{j\to\infty}\frac{1}{n_j}\sum_{k=0}^{n_j-1}\int_G\int_G\log\frac{\norm{gg'v_j}}{\norm{g'v_j}}\dd\mu(g)\dd\mu^{*k}(g')\\
&=\lim_{j\to\infty}\frac{1}{n_j}\int_G\log\frac{\norm{gv_j}}{\norm{v_j}}\dd\mu^{*n_j}(g)\le 0,
\end{align*}
the application of weak* convergence being justified since the function 
\begin{align*}
\R v\mapsto\int_G\log\frac{\norm{gv}}{\norm{v}}\dd\mu(g)
\end{align*}
on $\Pro(\R^d)$ is continuous by dominated convergence in view of the finite first moment assumption on $\mu$. 
Consequently, there exists a $\mu$-ergodic component $\nu$ of $\tilde{\nu}$ with support in $P$ satisfying 
\begin{align*}
\alpha(\nu)=\int_{\Pro(\R^d)}\int_G\log\frac{\norm{gv}}{\norm{v}}\dd\mu(g)\dd\nu(\R v)\le 0.
\end{align*}
The last statement in Theorem~\ref{thm:furstenberg_kifer} therefore implies $\nu(\Pro(F^{\leqs 0}))=1$. 
However, this is a contradiction to (iii), since as remarked before, this condition means that $P\cap\Pro(F^{\leqs 0})=\emptyset$. 
\end{proof}
\subsection{Expansion on Grassmannians}\label{subsec:grass}
Here, we introduce our main expansion assumption and show that it is satisfied in the settings of Theorems~\ref{thm:BQ_intro} and~\ref{thm:SW_main}. 

Let $\Gamma$ be a lattice in the real Lie group $G$ and $X=G/\Gamma$. 
In~\cite{EL}, Eskin--Lindenstrauss introduce the uniform expansion assumption for the adjoint representation to obtain a description of the $\mu$-ergodic $\mu$-stationary probability measures on $X$ (see \cite[Theorem~1.7]{EL}). 
However, as they point out, this condition is not sufficient to ensure that all such measures on $X$ are homogeneous. 
Below, we single out a stronger expansion assumption which guarantees that the only $\mu$-ergodic $\mu$-stationary probability measures on $X$ are finite periodic orbit measures and the Haar measure $m_X$. 

Let $V$ be a real vector space of dimension $d$. 
For each $1\le k\le d$, denote by $\Gr_k(V)$ the $k$-Grassmann variety of $V$. 
Let $\Gr_k(V)\hookrightarrow\Pro(V^{\wedge k})$ be the Pl\"ucker embedding. 
Its image is a closed subset of $\Pro(V^{\wedge k})$ given by $\Pro(\pt^k V)$, where we denote by $\pt^k V$ the set of non-zero pure wedge products in $V^{\wedge k}$. 
For a probability measure $\mu$ on $\GL_d(\R)$, we denote by $\bigwedge^k_*\mu$ the pushforward of $\mu$ under the $k^{\text{th}}$ exterior power representation. 
Note that all the $\bigwedge^k_*\mu$ have finite first moments if $\mu$ does, by virtue of the inequality $N(g^{\wedge k})\le N(g)^k$. 

\begin{definition}[Expansion on Grassmannians]
We say that a probability measure $\mu$ on $\GL_d(\R)$ is \emph{uniformly expanding on Grassmannians} if $\bigwedge^k_*\mu$ is uniformly expanding on $\Pro(\pt^k \R^d)\subset\Pro(\bigwedge^k \R^d)$ for every $1\le k\le d-1$. 
\end{definition}

We will usually impose this expansion condition on $\Ad_*\mu$. 
This accounts for the cases previously studied by Benoist--Quint in~\cite{BQ1} (Proposition~\ref{prop:BQ_assumptions}) and Simmons--Weiss~\cite{SW} (Proposition~\ref{prop:SW_assumptions}). 
\begin{proposition}\label{prop:BQ_assumptions}
Let $G$ be a real Lie group with non-compact simple identity component such that the Zariski closure $\mathcal{G}$ of $\Ad(G)$ is Zariski connected. 
Suppose that $\mu$ has a finite first moment in $\g$ and that $\Ad(G_\calS)$ is Zariski dense in $\mathcal{G}$. 
Then $\Ad_*\mu$ is uniformly expanding on Grassmannians. 
\end{proposition}
We remark that in the statement above, one cannot relax the requirement of simplicity to semisimplicity. 
Indeed, expansion fails for any vector corresponding under the Pl\"ucker embedding to a non-trivial proper Lie ideal in $\g$. 
\begin{proof}
We are given that $\Ad(G_\calS)$ is Zariski dense in the non-compact simple real algebraic subgroup $\mathcal{G}$ of $\Aut(\g)$. 
From Furstenberg's theorem on positivity of the top Lyapunov exponent (see \cite[Theorem~8.6]{F}) it follows that $\Ad_*\mu$ is uniformly expanding in every finite-dimensional algebraic representation $(V,\rho)$ of $\mathcal{G}$ without fixed vectors. 
Indeed, using complete reducibility one may assume that $\Ad(G_\calS)$ acts irreducibly, which already implies strong irreducibility in view of Zariski connectedness of $\mathcal{G}$. 
Applying Theorem~\ref{thm:furstenberg_kifer} to the $k^{\text{th}}$ exterior power of the standard representation for some $1\le k\le \dim(G)-1$, we find that $F^{\leqs 0}$ consists of $\mathcal{G}$-fixed vectors only. 
Since a fixed element of $\pt^k\g$ would give rise to a non-trivial proper Lie ideal of $\g$, we conclude $\pt^k\g\cap F^{\leqs 0}=\emptyset$, which is uniform expansion on $\Pro(\pt^k\g)$. 
\end{proof}
\begin{proposition}\label{prop:SW_assumptions}
Suppose that $\mu$ has a finite first moment in $\g$ and satisfies conditions \textup{(I')} and \textup{(III')} from~\textup{\S\ref{subsec:iid_intro}}. 
Then $\Ad_*\mu$ is uniformly expanding on Grassmannians. 
\end{proposition}
\begin{proof}
Let $1\le k\le \dim(G)-1$ and apply Theorem~\ref{thm:furstenberg_kifer} to the $k^{\text{th}}$ exterior power of the adjoint representation. 
The obtained spaces $F_i$ are $G_\calS$-invariant. 
Since conditions (I') and (III') together force every invariant subspace to contain vectors exhibiting almost sure exponential growth, all the numbers $\beta_i(\mu)$ are positive, which is uniform expansion on $\Pro(\g^{\wedge k})$. 
\end{proof}
%\begin{remark}\label{rmk:SW_simplification2}
%We note in passing that given Simmons--Weiss' conditions (I) and (III), the argument in the proof above also implies almost sure maximal expansion on $\Pro(\g)$. 
%Consequently, we see that \cite[Proposition~3.7]{SW} also follows from Furstenberg--Kifer's theorem. 
%\end{remark}
Let us now explain the example at the end of~\S\ref{subsec:iid_intro} in greater detail. 
\begin{example}
Let $G=\SL_3(\R)$, $\Gamma=\SL_3(\Z)$ and $\mu=\tfrac13(\delta_{g_1}+\delta_{g_2}+\delta_{g_3})$ with 
\begin{align*}
g_1=\begin{pmatrix}3&&\\&2&\\&&1/6\end{pmatrix},\,g_2=\begin{pmatrix}3&&1\\&2&\\&&1/6\end{pmatrix}\text{ and }g_3=\begin{pmatrix}3&&\\&2&1\\&&1/6\end{pmatrix}.
\end{align*}
A calculation shows that the subspaces 
\begin{align*}
V^{++}=\set*{\begin{pmatrix}0&&t\\&0&\\&&0\end{pmatrix}\for t\in\R},\,V^+=\set*{\begin{pmatrix}0&&\\&0&t\\&&0\end{pmatrix}\for t\in\R},
\end{align*}
of $\g$ are $G_\calS$-invariant with Lyapunov exponent $\log(18)$ on $V^{++}$ and $\log(12)$ on $V^+$. 
Thus there cannot exist a subspace $W\subset \g$ satisfying (I) and (III). 
However, the space $W'=V^{++}\oplus V^+$ satisfies (I') and (III'), as can be verified by direct computation. 
More generally, the space $W_k$ for $k\ge 1$ can be defined as the sum of the eigenspaces of $\Ad(g_1)^{\wedge k}$ corresponding to eigenvalues strictly greater than $1$. 
That these spaces have the correct properties is established as in the proof of \cite[Theorem~6.4]{SW}. 
\end{example}
\subsection{Measure Classification and Equidistribution Under Expansion}\label{subsec:iid_results}
We are now ready to establish equidistribution under the assumption of uniform expansion on Grassmannians in the adjoint representation. 

As outlined in~\S\ref{subsec:iid_intro}, the first step is the classification of stationary measures. 
The result is essentially a corollary of Eskin--Lindenstrauss' classification in~\cite{EL}. 
As already indicated at the beginning of~\S\ref{subsec:grass}, the aspect that is new is that our stronger expansion condition allows to rule out exceptional stationary measures that can a priori occur in \cite[Theorem~1.7]{EL}. 
\begin{theorem}\label{thm:exp_stat}
Let $G$ be a real Lie group, $\Gamma$ a discrete subgroup of $G$, and $\nu$ a $\mu$-ergodic $\mu$-stationary probability measure on $X=G/\Gamma$. 
Suppose that $\mu$ has a finite first moment in $\g$, that $\Ad_*\mu$ is uniformly expanding on Grassmannians, and that $G_\calS$ acts transitively on the connected components of $X$. 
Then either
\begin{enumerate}
\item $\nu$ is $G_\calS$-invariant and supported on a finite $G_\calS$-orbit, or 
\item $\Gamma$ is a lattice and $\nu$ is the Haar measure $m_X$ on $X$. 
\end{enumerate}
\end{theorem}
The proof combines ideas from the proofs of \cite[Theorem~1.3]{EL} and \cite[Proposition~3.2]{SW}. 
\begin{proof}
In view of Proposition~\ref{prop:expansion_equivalence}, we may apply \cite[Theorem~1.7]{EL} with trivial $Z$. 
The conclusion is that if we are not in case (i), $\nu$ must be of the form 
\begin{align*}
\nu=\int_{G/H}g_*\nu_0\dd\lambda(g),
\end{align*}
where $H$ is a closed subgroup of $G$ of positive dimension, $\nu_0$ is an $H$-homogeneous probability measure on $X$, and $\lambda$ is a $\mu$-stationary probability measure on $G/H$. 
Observe that $H$ is unimodular, since $\nu_0$ being $H$-homogeneous implies that $H$ intersects a conjugate of $\Gamma$ in a lattice. 

If $\dim(H)=\dim(G)$, then $\lambda$ is $G_\calS$-invariant (being stationary on a countable set, see \cite[Lemma~8.3]{BQ1}), and since $G_\calS$ acts transitively on the connected components of $X$ by assumption, it follows that $\nu=m_X$. 

Otherwise, we have $k\coloneqq\dim(H)<\dim(G)$. 
Let $v_1,\dots,v_k$ be a basis of $\Lie(H)$ and consider $\rho=v_1\wedge\dots\wedge v_k\in\pt^k\g$ and the stabilizer $L=\Stab_G(\rho)$ of $\rho$ in $G$. 
Since $H$ is unimodular we have $H\subset L$. 
Thus $\lambda$ projects to a $\mu$-stationary probability measure $\hat{\lambda}$ on $G/L\cong G\rho\subset \g^{\wedge k}\setminus\set{0}$. 
The measure $\beta\otimes\hat{\lambda}$ is then a probability measure on $B\times \g^{\wedge k}$ preserved by the skew-product transformation 
\begin{align*}
\hat{T}(b,w)=(Tb,\Ad^{\wedge k}(b_0)w),
\end{align*}
where $b=(b_n)_n$ and $T$ is the shift on $B$ (see \cite[Proposition~2.14]{BQ_book}). 
However, since $\hat{\lambda}(\set{0})=0$, our expansion assumption implies that almost every trajectory under this transformation is divergent, contradicting Poincar\'e recurrence. 
\end{proof}
\begin{remark}\label{rmk:expansion}
To apply \cite[Theorem~1.7]{EL} in the proof above, we need uniform expansion on $\g$. 
In the exterior powers of $\g$ the proof only uses almost sure divergence, i.e.\ that for every $v\in\pt^k\g$ with $2\le k\le \dim(G)-1$ we have 
\begin{align*}
\lim_{n\to\infty}\norm{\Ad^{\wedge k}(g_{b|_n})v}=\infty
\end{align*}
for $\beta$-a.e.\ $b\in B$. 
However, this property in fact already implies uniform expansion. 

To see this, note that if uniform expansion does not hold, then the compact set $\Pro(\pt^k\g)\cap\Pro(F^{\leqs 0})$ is non-empty and $G_\calS$-invariant and therefore supports a $\mu$-ergodic $\mu$-stationary probability measure $\nu$. 
Using Atkinson/Kesten's lemma (see e.g.\ \cite[Lemma~II.2.2]{BL_book}) the above almost sure divergence implies $\alpha(\nu)>0$, which gives a contradiction in view of Theorem~\ref{thm:furstenberg_kifer}. 
\end{remark}
The second ingredient is non-escape of mass. 
\begin{proposition}\label{prop:non_escape}
Let $G$ be a real Lie group with simple identity component such that the Zariski closure of $\Ad(G)$ is Zariski connected and $\Gamma$ a lattice in $G$. 
Suppose that $\mu$ has finite exponential moments in $\g$ and that $\Ad_*\mu$ is uniformly expanding on Grassmannians. 
Then, almost surely, there is no escape of mass for the random walk on $X=G/\Gamma$, in the sense that for every $x_0\in X$ and $\epsilon>0$ there exists a compact set $K\subset X$ such that, $\beta$-a.s., 
\begin{align*}
\limsup_{n\to\infty}\tfrac{1}{n}\abs{\set{0\le k<n\for g_{b|_k}x_0\notin K}}\le \epsilon.
\end{align*}
\end{proposition}
As is by now standard (see e.g.\ \cite{BQ_rec,BQ3,EM}), recurrence results of this type are most conveniently established by constructing what is known as \emph{Lyapunov function} for the random walk (also referred to as \emph{Margulis function} in this context), that is, a proper continuous function $f\colon X\to[0,\infty)$ which is contracted by $\mu$ in the sense that there are constants $c<1$ and $d\ge 0$ such that
\begin{align*}
\int_Gf(gx)\dd\mu(g)\le cf(x)+d
\end{align*}
for all $x\in X$. 
The proof of the proposition above will thus boil down to showing the existence of such a Lyapunov function. 
Specifically, we are going to show that our expansion assumption allows using the construction of Eskin--Margulis in~\cite{EM}, a strategy that already appeared in the proof of \cite[Theorem~2.1]{SW}. 
\begin{proof}[Proof of Proposition~\textup{\ref{prop:non_escape}}]
By \cite[Lemma~3.10]{BQ3}, it is enough to exhibit a Lyapunov function for the random walk. 
In order to use results from~\cite{EM}, we need to perform some initial reductions. 

Setting $R=\ker(\Ad)$, we know that $R\cap \Gamma$ has finite index in $R$ and the image $\Ad(\Gamma)$ is a lattice in the Zariski closure $\mathcal{G}$ of $\Ad(G)$ (see \cite[Lemma~6.1]{BQ_rec}). 
Accordingly, the induced map  from $X=G/\Gamma$ to $\mathcal{G}/\Ad(\Gamma)$ is proper. 
We may thus assume to begin with that $G$ is a Zariski connected simple real algebraic group. 
If $\Gamma$ is cocompact, there is nothing to prove. 
So we may moreover assume that $\Gamma$ is nonuniform, finally placing us in the setting of~\cite{EM}. 

We want to use the construction of a Lyapunov function given in~\cite[\S3]{EM}. 
For this, what remains to argue is that \enquote{condition (A)}, formulated at the end of~\cite[\S2]{EM}, is satisfied. 
The requirement is that for certain representations $(\rho_i,V_i)$ of $G$ and vectors $w_i\in V_i$, the following contraction property holds: 
For sufficiently small $\delta>0$ there ought to exist $c<1$ and $n\in\N$ such that 
\begin{align}\label{condition_A}
\int_G\norm{\rho_i(g)v}^{-\delta}\dd\mu^{*n}(g)\le c\norm{v}^{-\delta}
\end{align}
for all $v\in Gw_i$. 
The representations $\rho_i\colon G\to\GL(V_i)$ and vectors $w_i$ occurring in the above condition are characterized by the property that the stabilizer of $\R w_i$ in $G$ is some predetermined maximal parabolic subgroup $P_i$ of $G$. 
In our case, we can thus take $V_i=\g^{\wedge\dim(P_i)}$, $\rho_i\colon G\to\SL(V_i)$ the respective exterior power of $\Ad$, and $w_i$ to be a volume form of the Lie algebra $\mathfrak{p}_i$ of $P_i$ (see \cite[Proposition~7.83(b)]{Kn_book}). 
However, using uniform expansion as input, the proof of \cite[Lemma~4.2]{EM} precisely shows that \eqref{condition_A} holds for all non-zero pure wedge products $v$. 
This finishes the proof. 
\end{proof}
Combining the previous statements, we arrive at the main equidistribution result of this section. 
\begin{theorem}\label{thm:exp_equidist}
Let $G$ be a real Lie group with simple identity component such that the Zariski closure of $\Ad(G)$ is Zariski connected and $\Gamma$ a lattice in $G$. 
Suppose that $G_\calS$ is not virtually contained in any conjugate of $\Gamma$, $G_\calS$ acts transitively on the connected components of $X=G/\Gamma$, $\mu$ has finite exponential moments in $\g$, and $\Ad_*\mu$ is uniformly expanding on Grassmannians. 
Then for every $x_0\in X$, the random walk trajectory $(g_{b|_n}x_0)_n$ equidistributes towards $m_X$ for $\beta$-a.e.\ $b\in B$. 
\end{theorem}
\begin{proof}
The remaining argument is standard: 
\begin{itemize}
\item The Breiman law of large numbers (see \cite[Corollary~3.3]{BQ3}) applied to the one-point compactification of $X$ shows that, almost surely, any weak* limit of the sequence $\frac1n\sum_{k=0}^{n-1}\delta_{g_{b|_k}x_0}$ of empirical measures is $\mu$-stationary. 
\item Non-escape of mass (Proposition~\ref{prop:non_escape}) implies that all such weak* limits are probability measures on $X$. 
\item Since there are no finite orbits, using the classification of stationary measures (Theorem~\ref{thm:exp_stat}) we conclude that all ($\mu$-ergodic components of) these limits coincide with the Haar measure $m_X$. 
Hence the result.\qedhere
\end{itemize}
\end{proof}
\begin{proof}[Proof of Theorem~\textup{\ref{thm:intro1}}]
Using connectedness of $G$, we see that the conditions of Theorem~\ref{thm:exp_equidist} are satisfied. 
\end{proof}
\section{Markov Random Walks}\label{sec:markov_walks}
We now turn our attention to Markov random walks. 
We first adopt a bootstrapping approach (\S\ref{subsec:markov_stat_bootstrap},~\S\ref{subsec:markov_equidist_bootstrap}), upgrading statements about the random walk with i.i.d.\ increments $Z_n=Y_{\tau_g^{n+1}-1}\dotsm Y_{\tau_g^n}$ to statements about the whole random walk. 
As preparation, we study the distribution of these excursions, which we call \emph{renewal measures}, in~\S\ref{subsec:renewal_meas}. 
In~\S\ref{subsec:markov_mom_exp} we discuss expansion in the Markovian setting, and in~\S\ref{subsec:exp_markov} we prove our main result (Theorem~\ref{thm:exp_markov_equidist}) about expanding Markov chains, which implies Theorem~\ref{thm:intro2}. 
The final subsection~\S\ref{subsec:upper_block_form} is dedicated to a concrete example that contains Corollary~\ref{cor:intro4} and will be important in~\S\ref{sec:fractals} about Diophantine approximation on fractals. 

Let $E$ be a countable set. 
A Markov chain on $E$ is defined by a transition kernel $P$ on $E$. 
Recall that this means that for every $e\in E$, $P(e,\cdot)$ is a probability distribution on $E$. 
We are going to write $p_{e',e}\coloneqq P(e,e')$ for the probability of going from state $e$ to $e'$, and $p_w=p_{e_{n-1},e_{n-2}}\dotsm p_{e_1,e_0}$ for a word $w=e_{n-1}\ldots e_0\in E^*$. 
(Since we are studying random walks on $X$ coming from a left action of $G$, we are using a right-to-left ordering throughout this section.) 
Let $\Pro_e$ be the associated Markov measure on $\Omega=E^\N$ starting at $e\in E$ (at time $n=0$), characterized by the property that 
\begin{align*}
\Pro_e\Bigl[\set{e_0}\times\dots\times\set{e_n}\times E^\N\Bigr]=p_{e_n,e_{n-1}}\dotsm p_{e_1,e_0}\delta_{e_0=e}
\end{align*}
for $e_0,\dots,e_n\in E$. 
More generally, for an arbitrary probability distribution $\lambda$ on $E$ we write 
\begin{align}\label{arb_markov}
\Pro_\lambda=\sum_{e\in E}\lambda(\set{e})\Pro_e,
\end{align}
which is the unique Markov measure on $\Omega$ for the given Markov chain on $E$ with starting distribution $\lambda$. 
Expectation with respect to the probability measures $\Pro_e$ and $\Pro_\lambda$ will be denoted by $\E_e$ and $\E_\lambda$, respectively. 

The consecutive hitting times of a state $e$ will be denoted by $\tau_e^n$, defined by $\tau_e^0=0$ and 
\begin{align*}
\tau_e^n(\omega)=\inf\set{n>\tau_e^{n-1}(\omega)\for \omega_n=e}
\end{align*}
for $\omega=(\omega_m)_m\in\Omega$ and $n\in\N$. 
We abbreviate the first hitting time $\tau_e^1$ as $\tau_e$. 

We will only be interested in irreducible chains, i.e.\ ones where every state can be reached from every other in finite time with positive probability (formally, chains with $\Pro_e[\tau_{e'}<\infty]>0$ for any two states $e,e'\in E$). 
Let us recall the classical notions of recurrence for Markov chains. 
\begin{definition}\label{def:markov_states}
An irreducible Markov chain on $E$ is called 
\begin{itemize}
\item \emph{recurrent} if $\Pro_e[\tau_e<\infty]=1$ for every $e\in E$, 
\item \emph{positive recurrent} if $\E_e[\tau_e]<\infty$ for every $e\in E$, and 
\item \emph{exponentially recurrent} if for every $e\in E$ there exists $\delta>0$ such that $\E_e[\exp(\delta\tau_e)]<\infty$. 
\end{itemize}
\end{definition}
It is well known that these forms of recurrence hold for all states as soon as one state has the respective property. 
%For (positive) recurrence this is classical (see e.g.\ \cite[Corollary~4.2,~Theorem~6.2]{Ch_book}); for exponential recurrence it was established by Kendall (\cite[Lemma~7.1]{Ke}). 
Irreducible positive recurrent chains admit a unique stationary probability distribution $\pi$ on $E$, %(\cite[Theorem~7.1]{Ch_book}), 
given by 
\begin{align}\label{stat_dist}
\pi(\set{e'})=\frac{1}{\E_e[\tau_e]}\E_e\sqbr*{\sum_{k=0}^{\tau_e-1}\mathbf{1}_{\omega_k=e'}}
\end{align}
for $e,e'\in E$. %(\cite[before~Corollary~9.1]{Ch_book}). 
%In particular, for $e\in E$ we have 
%\begin{align*}
%\pi(\set{e})=\frac{1}{\E_e[\tau_e]}>0.
%\end{align*}
The Markov measure $\Pro_\pi$ is then invariant and ergodic under the shift map $T$ on $\Omega$. 
See e.g.\ Chung~\cite{Ch_book} for proofs of these classical facts. 

For the sequel, we fix a coding map $E\ni e\mapsto g_e\in G$. 
Such a map allows us to define a stochastic process on $G$ by 
\begin{align*}
(Y_n)_n\colon\Omega\ni\omega\mapsto (g_{\omega_{n-1}})_n.
\end{align*}
This process is generally not a Markov chain on $G$ in the usual sense. 
Indeed, denoting by $\calS\subset G$ the image of the coding map, any generalized Markov measure on $\calS^\N$ can be obtained in this manner as the distribution of $(Y_n)_n$. 
Similarly, the induced random walk $(Y_n\dotsm Y_1x_0)_n$ on $X$ does not constitute a Markov chain. 
However, this flaw can be removed by embedding this random walk into a Markov chain on a larger space. 
\begin{definition}\label{def:action_chain}
Given a Markov chain on $E$ with transition kernel $P$, the \emph{action chain} is the Markov chain on $E\times X$ defined by the transition kernel $Q$ given by 
\begin{align*}
Q(e,x)=P(e,\cdot)\otimes \delta_{g_ex}
\end{align*}
for $(e,x)\in E\times X$. 
\end{definition}
The interpretation is that the $E$-coordinate contains the element to be applied next. 
A step into the future consists of the application of that group element to the $X$-coordinate and choosing the next element according to the transition kernel $P$ in the $E$-coordinate. 

It is evident by construction that in the $X$-coordinate of the action chain we obtain our random walks of interest of the form $(Y_n\dotsm Y_1x_0)_n$. 
A precise formulation of this statement is the following. 
\begin{lemma}
Let $\lambda$ be any distribution on $E$ and $x_0\in X$. 
Write $\Pro_{\lambda\otimes \delta_{x_0}}$ for the Markov measure on $(E\times X)^\N$ for the action chain starting from $\lambda\otimes \delta_{x_0}$, and $\operatorname{pr}_E\colon (E\times X)^\N\to E^\N$, $\operatorname{pr}_X\colon (E\times X)^\N\to X^\N$ for the projections onto all the $E$- and $X$-coordinates, respectively. 
Then the pushforward of $\Pro_{\lambda\otimes\delta_{x_0}}$ by $\operatorname{pr}_E$ is $\Pro_\lambda$ and the pushforward by $\operatorname{pr}_X$ is the distribution of $(Y_n\dotsm Y_1x_0)_n$, where $(Y_n)_n$ is the Markov chain on $E$ starting with $Y_1$ distributed according to $\lambda$. \qed
\end{lemma}
%\begin{proof}
%That $(\operatorname{pr}_E)_*\Pro_{\lambda\otimes\delta_{x_0}}=\Pro_\lambda$ is immediate from the definition of $Q$. 
%The other statement follows from this and the fact that the projection $\operatorname{pr}_X$ coincides $\Pro_{\lambda\otimes\delta_{x_0}}$-a.s.\ with the composition $a\circ\operatorname{pr}_E$, where $a\colon E^\N\to X^\N$ is the map 
%\begin{align*}
%\omega\mapsto (g_{\omega|_n}x_0)_n=(g_{\omega_{n-1}}\dotsm g_{\omega_0}x_0)_n.&\qedhere
%\end{align*}
%\end{proof}
\subsection{Renewal Measures}\label{subsec:renewal_meas}
We say that a word $w=e_{n-1}\ldots e_0\in E^*$ is \emph{\adm{e}{e'}-admissible} if $e_0=e$, $p_{e_k,e_{k-1}}>0$ for $1\le k\le n-1$ and $p_{e',e_{n-1}}>0$, and simply that it is \emph{admissible} if it is \adm{e}{e'}-admissible for some $e,e'\in E$. 
Further, we call $w$ an \emph{$e$--renewal word} if it is \adm{e}{e}-admissible and $e_k\neq e$ for $1\le k\le n-1$, and denote the set of $e$--renewal words by $E_e^\ren$. 
A sequence $\omega\in\Omega=E^\N$ is said to be admissible if all words $\omega|_n$ for $n\in\N$ are. 
The set of all admissible sequences is going to be denoted $E^\infty$, and $E_e^\infty$ is the subset of such sequences starting with $e$. 
\begin{definition}\label{def:renewal_meas}
Given a recurrent irreducible Markov chain on $E$ and a state $e\in E$, we define the measure $\tilde{\mu}_e$ on the set $E_e^\ren$ of $e$--renewal words by 
\begin{align*}
\tilde{\mu}_e(\set{w})&\coloneqq p_{e,e_{n-1}}p_w=p_{e,e_{n-1}}p_{e_{n-1},e_{n-2}}\dotsm p_{e_1,e}\\
&=\Pro_e\Bigl[\set{\omega\in\Omega\for\omega_1=e_1,\dots,\omega_{n-1}=e_{n-1},\tau_e(\omega)=n}\Bigr]
\end{align*}
for $w=e_{n-1}\ldots e_1e\in E_e^\ren$. 
Then the \emph{renewal measure} $\mu_e$ starting at $e\in E$ is defined to be the pushforward  of $\tilde{\mu}_e$ to $G$ via the coding map $E^*\to G,\,w\mapsto g_w$. 
\end{definition}
Note that recurrence implies that $\Pro_e$-a.s.\ we have $\tau_e<\infty$, so that under this assumption $\tilde{\mu}_e$ and $\mu_e$ are probability measures. 

The following simple lemma formalizes the fact that consecutive excursions of a Markov chain are i.i.d. 
\begin{lemma}\label{lem:isom}
For a recurrent irreducible Markov chain on $E$ and any state $e\in E$, the map 
\begin{align}\label{isom}
(E_e^\infty,\Pro_e)&\to\Bigl((E_e^\ren)^\N,\tilde{\mu}_e^{\otimes\N}\Bigr) \\
\omega&\mapsto\Bigl(\omega_{\tau_e^{n+1}-1}\ldots\omega_{\tau_e^n}\Bigr)_n\nonumber
\end{align}
is an isomorphism \textup{(}mod $0$\textup{)} of probability spaces. 
\end{lemma}
\begin{proof}
On the $\Pro_e$--full measure subset 
\begin{align*}
\set{\omega\in E_e^\infty\for \tau_e^n(\omega)<\infty\text{ for all }n\in\N}
\end{align*}
of $\Omega$ the given map is a well-defined bijection. 
To see that it is measure-preserving it suffices to consider cylinder sets of the form $\set{w_0}\times \dots\times \set{w_N}\times (E_e^\ren)^\N$ for $e$--renewal words $w_0,\dots,w_N$. 
But for such sets the statement follows directly by construction of $\tilde{\mu}_e$. 
\end{proof}
Before moving on, let us shed some light on the relationship between the various renewal measures $\mu_e$ on $G$, knowledge of which will be of interest later on. 
For this, we denote by $G_e^+$ (resp.\ $G_e$) the closed subsemigroup (resp.\ subgroup) of $G$ generated by the support of $\mu_e$, and by $G_\calS$ the image of the set of admissible words under the coding map $E\to G$. 
Note that $G_\calS$ is in general not closed under multiplication. 
In case $G$ is real algebraic, we write $H_e$ and $H_\calS$ for the Zariski closures of $G_e^+$ and $G_\calS$, respectively. 
\begin{lemma}\label{lem:Z_closures_conjugate}
Assume the Markov chain on $E$ is irreducible and recurrent. 
\begin{enumerate}
\item Let $c\in E^*$ be \adm{e}{e'}-admissible and $c'\in E^*$ be \adm{e'}{e}-admissible. Then the semigroups $G_e^+$ and $G_{e'}^+$ satisfy 
\begin{align*}
g_{c'}G_{e'}^+g_c\subset G_e^+.
\end{align*}
\end{enumerate}
If $G$ is real algebraic, we additionally have the following. 
\begin{enumerate}
\setcounter{enumi}{1}
\item The groups $H_e$ and $H_{e'}$ are conjugate inside $H_\calS$. More precisely, with $c,c'$ as in \textup{(i)} we have 
\begin{align*}
g_{c'}H_{e'}g_{c'}^{-1}=g_c^{-1}H_{e'}g_c=H_e.
\end{align*}
\item If there exists $\tilde{e}\in E$ with both $e\tilde{e}$ and $e'\tilde{e}$ admissible, then $H_e=H_{e'}$. 
\item If there exists $\tilde{e}\in E$ with $e\tilde{e}$ admissible for all $e\in E$, then all $H_e$ coincide and $H_\calS$ is contained in their normalizer. 
\item If there exists $\tilde{e}\in E$ with $\tilde{e}e$ admissible for all $e\in E$, then $H_e=H_\calS$ for all $e\in E$. In particular, $H_\calS$ is a group. 
\end{enumerate}
\end{lemma}
\begin{proof}
For (i), simply note that for every \adm{e'}{e'}-admissible word $w\in E^*$ the word $c'wc$ is \adm{e}{e}-admissible. 

For (ii), taking the Zariski closure of both sides of the inclusion in (i), we get $g_{c'}H_{e'}g_c\subset H_e$. 
Since the word $cc'$ is \adm{e'}{e'}-admissible, $g_{c'}G_{e'}^+g_c$ is a semigroup and hence its Zariski closure $g_{c'}H_{e'}g_c$ is a subgroup of $H_e$. 
This implies
\begin{align*}
g_{c'}H_{e'}g_c=g_{c'}H_{e'}g_{c'}^{-1}=g_c^{-1}H_{e'}g_c\subset H_e.
\end{align*}
By the symmetric argument, we also have $g_cH_eg_c^{-1}\subset H_{e'}$, which in combination with the above gives (ii). 

For (iii) note that existence of such an element $\tilde{e}$ implies that $c'$ can be chosen to be both \adm{e'}{e'}- and \adm{e'}{e}-admissible. 
Then $g_{c'}\in H_{e'}$ and we conclude using~(ii). 

In (iv), all the $H_e$ coincide due to (iii). 
For every admissible word $w\in E^*$, the word $w\tilde{e}$ is \adm{\tilde{e}}{e}-admissible for some $e\in E$. 
Thus, using $g_{\tilde{e}}\in H_{\tilde{e}}$ and part (ii) we find 
\begin{align*}
g_wH_{\tilde{e}}g_w^{-1}=g_{w\tilde{e}}H_{\tilde{e}}g_{w\tilde{e}}^{-1}=H_e=H_{\tilde{e}}.
\end{align*}
This shows that $G_\calS$ is contained in the normalizer of $H_{\tilde{e}}$. 
Hence, so is $H_\calS$. 

In the setting of (v), $H_e=H_{\tilde{e}}$ for all $e\in E$ again follows from (iii). 
Clearly, we also have $H_{\tilde{e}}\subset H_\calS$. 
For the reverse inclusion let $w\in E^*$ be any admissible word, say \adm{e}{e'}-admissible, and choose a \adm{\tilde{e}}{e}-admissible word $c\in E^*$. 
Then both $c$ and $wc$ are \adm{\tilde{e}}{\tilde{e}}-admissible. 
This implies $g_w\in H_{\tilde{e}}$, and hence $H_\calS\subset H_{\tilde{e}}$. 
\end{proof}
\subsection{Stationary Measures}\label{subsec:markov_stat_bootstrap}
Next, we describe the structure of ergodic stationary measures for the action chain in terms of ergodic stationary measures for the renewal measures $\mu_e$. 
\begin{lemma}\label{lem:markov_stat_bootstrap}
Suppose the Markov chain on $E$ is irreducible and positive recurrent and let $\pi$ be its stationary distribution. 
If $\nu$ is a stationary probability measure for the action chain on $E\times X$, then 
\begin{align}\label{stat_decomp}
\nu=\sum_{e\in E}\pi(\set{e})\delta_e\otimes \nu_e,
\end{align}
where for each $e\in E$, $\nu_e$ is a $\mu_e$-stationary probability measure on $X$, satisfying 
\begin{align}\label{action_stat}
\pi(\set{e})\nu_e=\sum_{e'\in E}\pi(\set{e'})p_{e,e'}(g_{e'})_*\nu_{e'}.
\end{align}
If $\nu$ is ergodic, then the $\nu_e$ are $\mu_e$-ergodic. 

Furthermore, if $c\in E^*$ is \adm{e'}{e}-admissible, we have $(g_c)_*\nu_{e'}\ll\nu_e$, and if $\nu_e$ is $G_e^+$-invariant, then $\nu_e$ and $(g_c)_*\nu_{e'}$ belong to the same measure class. 
\end{lemma}
\begin{proof}
For any measurable subset $Y\subset X$ and $e\in E$ we have by stationarity of $\nu$ 
\begin{align*}
\pi(\set{e})\nu_e(Y)&=\nu(\set{e}\times Y)=\nu Q(\set{e}\times Y)=\int_{E\times X}Q\bigl((e',x),\set{e}\times Y\bigr)\dd\nu(e',x)\\
&=\sum_{e'\in E}\pi(\set{e'})p_{e,e'}\nu_{e'}(g_{e'}^{-1}Y),
\end{align*}
which is precisely \eqref{action_stat}. 
Specializing to $Y=X$ shows that the projection of $\nu$ to $E$ is a stationary probability measure for the abstract Markov chain on $E$. 
By uniqueness, it follows that this projection is $\pi$, or in other words that the $\nu_e$ are probability measures. 

The fact that the $\nu_e$ are $\mu_e$-stationary (and $\mu_e$-ergodic if $\nu$ is ergodic) follows from \cite[Lemma~3.4]{BQ3} applied to the $Q$-recurrent subsets $\set{e}\times X$ of $E\times X$. 

The first statement about absolute continuity follows by noting that as a consequence of \eqref{action_stat} and by induction, for every $e\in E$ and $n\in\N$ we have 
\begin{align*}
\pi(\set{e})\nu_e=\sum_{\substack{w=e_{n-1}\ldots e_0\in E^*\\\adm{e'}{e}\text{-admissible}}}\pi(\set{e'})p_{e,e_{n-1}}p_w(g_w)_*\nu_{e'},
\end{align*}
with all occurring factors positive. 
For the last claim let $c'\in E^*$ be \adm{e}{e'}-admissible. 
Then we have $g_{cc'}\in G_e^+$, so that the above and $G_e^+$-invariance of $\nu_e$ imply 
\begin{align*}
\nu_e=(g_{cc'})_*\nu_e\ll (g_c)_*\nu_{e'}\ll \nu_e.&\qedhere
\end{align*}
\end{proof}
\subsection{Equidistribution}\label{subsec:markov_equidist_bootstrap}
This subsection contains the joint equidistribution results alluded to in~\S\ref{subsec:markov_intro}, which represent a key ingredient of our approach. 
\begin{proposition}\label{prop:markov_equidist_bootstrap}
Suppose the Markov chain on $E$ is irreducible and positive recurrent and denote by $\pi$ its stationary distribution. 
Let $x_0\in X$, $e\in E$ and $m$ be a probability measure on $X$ invariant under $g_w$ for every admissible word $w\in E^*$ starting with $e$. 
If the trajectory $(g_{b|_n}x_0)_n$ equidistributes towards $m$ for $\mu_e^{\otimes \N}$-a.e.\ $b\in B$, then $(g_{\omega|_n}x_0,T^n\omega)_n$ equidistributes towards $m\otimes \Pro_\pi$ for $\Pro_e$-a.e.\ $\omega\in\Omega$. 
\end{proposition}

In the proof of Proposition~\ref{prop:markov_equidist_bootstrap} we will need part (i) of the following technical lemma. 
Part (ii) will be used in~\S\ref{sec:fractals} about Diophantine approximation on fractals. 
\begin{lemma}\label{lem:equidist_reduction}
\leavevmode
\begin{enumerate}
\item \textup{(}\cite[Proposition~5.1]{SW}\textup{)} Let $\Pro=\mu^\N$ be the Bernoulli measure on $\Omega$ associated to a probability measure $\mu$ on $E$. 
Assume that $(g_{\omega|_n}x_0)_n$ equidistributes towards a probability measure $m$ on $X$ for $\Pro$-a.e.\ $\omega\in\Omega$. 
Then $(g_{\omega|_n}x_0,T^n\omega)_n$ equidistributes towards $m\otimes \Pro$ for $\Pro$-a.e.\ $\omega\in\Omega$. 
\item Denote by $\pi$ the stationary distribution of a positive recurrent Markov chain on $E$ and let $\lambda$ be any starting distribution on $E$. 
Assume that $(\omega_n,g_{\omega|_n}x_0)_n$ equidistributes towards a probability measure on $E\times X$ of the form 
\begin{align*}
\sum_{e\in E}\pi(\set{e})\delta_e\otimes m_e
\end{align*}
for $\Pro_\lambda$-a.e.\ $\omega\in\Omega$. 
Then $(\omega_n,g_{\omega|_n}x_0,T^n\omega)_n$ equidistributes towards 
\begin{align*}
\sum_{e\in E}\pi(\set{e})\delta_e\otimes m_e\otimes \Pro_e
\end{align*}
for $\Pro_\lambda$-a.e.\ $\omega\in\Omega$. 
\end{enumerate}
\end{lemma}
The first part of this lemma is essentially contained in the article~\cite{SW} of Simmons--Weiss, whose proof relies on ideas going back to Kolmogorov and Doob (cf.\ \cite[\S A.3]{BQ_book}). 
Our proof of the second part generalizes the argument to the Markovian case. 

The method of proof is to show the desired almost sure convergence for a fixed test function and then use separability of an appropriate space of test functions to exchange the order of quantifiers. 
When the underlying space is locally compact, this test function space can be taken to be the space of compactly supported continuous functions. 
This is however not the case in our setup, so that we need to find a substitute. 
To this end, let us introduce the following concept: 
Given a locally compact second countable metrizable space $S$, we shall say that a continuous function $f$ on $S\times\Omega$ \emph{compactly depends on finitely many coordinates} if there exists $N\in\N$ and a compactly supported continuous function $\tilde{f}$ on $S\times E^{N+1}$ such that $f(s,\omega)=\tilde{f}(s,\omega_0,\dots,\omega_N)$ for all $(s,\omega)\in S\times \Omega$. 
The collection of all such functions is separable; we denote it by $C_{cf}(S\times \Omega)$. 
\begin{proof}[Proof of Lemma~\textup{\ref{lem:equidist_reduction}}]
It is easily deduced from \cite[Propositions~3.4.4,~3.4.6]{EK_book} that it suffices to check convergence for test functions $f\in C_{cf}(S\times \Omega)$, where $S=X$ in~(i) and $S=E\times X$ in (ii). 
In view of separability of this function space, the proof of \cite[Proposition~5.1]{SW} still yields part (i). 

Similarly, to obtain (ii) it is enough to establish the desired $\Pro_\lambda$-a.s.\ convergence 
\begin{align*}
\frac{1}{n}\sum_{k=0}^{n-1}\varphi(\omega_k,g_{\omega|_k}x_0,T^k\omega)\overset{n\to\infty}{\longrightarrow}\sum_{e\in E}\pi(\set{e})\int_{X\times \Omega}\varphi(e,x,\omega)\dd(m_e\otimes \Pro_e)
\end{align*}
for a single bounded continuous test function $\varphi\colon E\times X\times \Omega\to\R$ depending on finitely many coordinates, say on the first $N+1$ coordinates $\omega_0,\dots,\omega_N$ in $\Omega$. 

Introduce the functions 
\begin{align*}
\varphi_X(e,x)&=\int_\Omega\varphi(e,x,\omega)\dd\Pro_e(\omega),\text{ and}\\
h(e,x,\omega)&=\varphi(e,x,\omega)-\varphi_X(e,x).
\end{align*}
Applying $\Pro_\lambda$-a.s.\ equidistribution of $(\omega_n,g_{\omega|_n}x_0)_n$ to the function $\varphi_X$ and setting $z_k=z_k(\omega)=(\omega_k,g_{\omega|_k}x_0,T^k\omega)$, we see that it remains to show that $\Pro_\lambda$-a.s.\ we have 
\begin{align}\label{to_show_vanishes}
\frac1n\sum_{k=0}^{n-1}h(z_k)\overset{n\to\infty}{\longrightarrow}0.
\end{align} 
Denote by $\mathcal{B}_n$ the $\sigma$-algebra of Borel subsets of $E\times X\times \Omega$ depending only on the first $n+1$ coordinates $\omega_0,\dots,\omega_n$ in $\Omega$. 
Then by definition of $z_k$ and assumption on $\varphi$ we have for $k\le n-N$ 
\begin{align}\label{mart_ingr1}
\E_\lambda\sqbr{h(z_k)|\mathcal{B}_n}=h(z_k).
\end{align}
Now suppose $k\ge n$. 
Then, using the Markov property and the definition of $\varphi_X$, 
\begin{align}\label{eq.1}
\int_\Omega &\varphi_X(\omega_{k-n}',g_{\omega'|_{k-n}}g_{\omega|_n}x_0) \dd \Pro_{\omega_n}(\omega') \nonumber\\
&=\int_\Omega \int_\Omega \varphi(\omega_{k-n}',g_{\omega'|_{k-n}}g_{\omega|_n}x_0,\omega'')\dd \Pro_{\omega'_{k-n}}(\omega'')\dd \Pro_{\omega_n}(\omega')\nonumber\\
&=\int_\Omega \varphi(\omega_{k-n}',g_{\omega'|_{k-n}}g_{\omega|_n}x_0,T^{k-n}\omega') \dd \Pro_{\omega_n}(\omega').  
\end{align}

Using the Markov property again, one can express the conditional expectation $\E_\lambda\sqbr{h(z_k)|\mathcal{B}_n}$ as 
\begin{align}\label{eq.2}
\E_\lambda\sqbr{h(z_k)|\mathcal{B}_n}&=\E_\lambda\sqbr[\Big]{h(\omega_k,g_{\omega|_k}x_0,T^k\omega)\Big|\mathcal{B}_n}\nonumber\\
&=\int_\Omega h(\omega_{k-n}',g_{\omega'|_{k-n}}g_{\omega|_n}x_0,T^{k-n}\omega')\dd\Pro_{\omega_n}(\omega').
\end{align}

Combining \eqref{eq.1} and \eqref{eq.2}, we deduce that for $k \ge n$ we have 
\begin{align}\label{mart_ingr2}
    \E_\lambda\sqbr{h(z_k)|\mathcal{B}_n}=0.
\end{align}
It follows from \eqref{mart_ingr1} and \eqref{mart_ingr2} that the random variables 
\begin{align*}
M_n=\sum_{k=0}^\infty \E_\lambda\sqbr{h(z_k)|\mathcal{B}_n}
\end{align*}
form a martingale under $\Pro_\lambda$ differing by a bounded amount (at most $2N\norm{h}_\infty$) from $\sum_{k=0}^{n-1}h(z_k)$. 
In particular, $(M_n)_n$ has bounded increments, so that \cite[Corollary~A.8]{BQ_book} yields that $\Pro_\lambda$-a.s.\ $\frac1n M_n\to 0$ as $n\to\infty$, proving \eqref{to_show_vanishes} and hence the lemma. 
\end{proof}
\begin{proof}[Proof of Proposition~\textup{\ref{prop:markov_equidist_bootstrap}}]
For the sake of readability, we shall first ignore the second component $T^n\omega$ and only prove equidistribution of $(g_{\omega|_n}x_0)_n$. 
Afterwards, we explain the modifications needed to obtain the full statement. 

Let $f$ be a bounded continuous function on $X$. 
For $\ell\in\N$ we consider the function 
\begin{align*}
F_\ell\colon X\times (E_e^\ren)^\N\to\R,\,(x,(w_m)_m)\mapsto\begin{cases}f(g_{w_0|_\ell}x),&\ell(w_0)>\ell,\\\hfill 0,&\ell(w_0)\le \ell,\end{cases}
\end{align*}
where $\ell(w_0)$ denotes the length of the word $w_0$. 
Applying Lemma~\ref{lem:equidist_reduction}(i) to $F_\ell$ with $\Pro=\tilde{\mu}_e^{\otimes \N}$ and using the invariance assumption on $m$, we get 
\begin{align}\label{convergence_from_lemma}
\frac{1}{n}\sum_{k=0}^{n-1}f(g_{w_k|_\ell}g_{w_{k-1}}\dotsm g_{w_0}x_0)\mathbf{1}_{\ell(w_k)>\ell}\longrightarrow &\int F_\ell\dd\bigl(m\otimes\tilde{\mu}_e^{\otimes\N}\bigr)\\=&\,\Pro_e[\tau_e>\ell]\int f\dd m\nonumber
\end{align}
as $n\to\infty$ for $\tilde{\mu}_e^{\otimes \N}$-a.e.\ $(w_m)_m\in(E_e^\ren)^\N$. 

Now let $\omega\in\Omega$ correspond to $(w_m)_m\in (E_e^\ren)^\N$ via \eqref{isom} and denote by $T(n)$ the number of occurrences of $e$ in $\omega$ before time $n$. 
In other words, $T(n)$ is the number of the $w_m$ contributing to $\omega|_n$, so that the latter is some intermediate word between $w_{T(n)-2}\ldots w_0$ and $w_{T(n)-1}\ldots w_0$. 
Using this observation, for every $L\in\N$ we can write %\footnote{A similar splitting is used in~\cite[\S14,~\S15]{Ch_book} under the name \enquote{dissection formula}.}
\begin{align}
\frac{1}{n}\sum_{k=0}^{n-1}f(g_{\omega|_k}x_0)=&\,\frac{T(n)}{n}\sum_{\ell=0}^{L-1}\frac{1}{T(n)}\sum_{k=0}^{T(n)-1}f(g_{w_k|_\ell}g_{w_{k-1}}\dotsm g_{w_0}x_0)\mathbf{1}_{\ell(w_k)>\ell}\label{convergence_to_correct}\\
&+\frac{1}{n}\sum_{k=0}^{T(n)-1}\sum_{\ell=L}^{\ell(w_k)-1}f(g_{w_k|_\ell}g_{w_{k-1}}\dotsm g_{w_0}x_0)\label{negligible1}\\
&-\frac{1}{n}\sum_{k=n}^{\tau_e^{T(n)}(\omega)-1}f(g_{\omega|_k}x_0).\label{negligible2}
\end{align}
Using Lemma~\ref{lem:isom} and the Birkhoff ergodic theorem, we have $\Pro_e$-a.s.\  $\tau_e^n/n\to\E_e[\tau_e]$. 
This in turn implies that $\Pro_e$-a.s.\ also $T(n)/n\to 1/\E_e[\tau_e]$. %(cf.\ \cite[Corollary~15.1]{Ch_book}). 
Together with \eqref{convergence_from_lemma} it follows that the right-hand side of \eqref{convergence_to_correct} converges $\Pro_e$-a.s.\ to 
\begin{align*}
\frac{1}{\E_e[\tau_e]}\sum_{\ell=0}^{L-1}\Pro_e[\tau_e> \ell]\int f\dd m.
\end{align*}
Using the ergodic theorem again, we also know that \eqref{negligible1} is bounded by 
\begin{align*}
\frac{\norm{f}_\infty}{n}\sum_{k=0}^{T(n)-1}(\ell(w_k)-L)^+\overset{n\to\infty}{\longrightarrow} \frac{\norm{f}_\infty}{\E_e[\tau_e]}\E_e[(\tau_e-L)^+],
\end{align*}
and \eqref{negligible2} by 
\begin{align*}
\frac{\norm{f}_\infty\ell(w_{T(n)-1})}{n}\overset{n\to\infty}{\longrightarrow}0,
\end{align*}
where in both cases convergence holds $\Pro_e$-a.s. 

Since $\sum_{\ell=0}^\infty\Pro_e[\tau_e> \ell]=\E_e[\tau_e]$ and, by positive recurrence, $\E_e[(\tau_e-L)^+]\to 0$ as $L\to\infty$, the above combine to imply the desired $\Pro_e$-a.s.\ convergence 
\begin{align*}
\frac{1}{n}\sum_{k=0}^{n-1}f(g_{\omega|_k}x_0)\longrightarrow \int f\dd m
\end{align*}
as $n\to\infty$. 

We now upgrade the argument above to also obtain joint equidistribution. 
With the same initial reduction as in the proof of Lemma~\ref{lem:equidist_reduction}, it suffices to prove $\Pro_e$-a.s.\ convergence 
\begin{align*}
\frac{1}{n}\sum_{k=0}^{n-1}f(g_{\omega|_k}x_0,T^k\omega)\overset{n\to\infty}{\longrightarrow}\int f\dd(m\otimes \Pro_\pi)
\end{align*}
for one fixed bounded continuous function $f$ on $X\times \Omega$ depending on only finitely many coordinates. 
The argument is similar as above; only the functions $F_\ell$ need to be chosen in a slightly more intricate way: 
We set
\begin{align*}
F_\ell\colon X\times (E_e^\ren)^\N\to\R,\,(x,(w_m)_m)\mapsto\begin{cases}f(g_{w_0|_\ell}x_0,T_\ell(w_m)_m),&\ell(w_0)> \ell,\\\hfill 0,&\ell(w_0)\le \ell,\end{cases}
\end{align*}
where $T_\ell(w_m)_m$ is obtained by first identifying $(w_m)_m$ with $\omega\in\Omega$ via \eqref{isom} and then applying the $\ell$-fold shift $T^\ell$. 
These functions $F_\ell$ again satisfy the assumptions of part (i) of Lemma~\ref{lem:equidist_reduction}. 
We find 
\begin{align*}
\frac{1}{n}\sum_{k=0}^{n-1}f(g_{w_k|_\ell}g_{w_{k-1}}\dotsm g_{w_0}x_0,T_\ell(w_{m+k})_m)\mathbf{1}_{\ell(w_k)> \ell}&\longrightarrow\int F_\ell\dd\bigl(m\otimes\tilde{\mu}_e^{\otimes\N}\bigr)
\end{align*}
as $n\to\infty$ for $\tilde{\mu}_e^{\otimes \N}$-a.e.\ $(w_m)_m\in(E_e^\ren)^\N$, the limit equaling, again by the assumed invariance of $m$ and Lemma~\ref{lem:isom}, 
\begin{align*}
\int F_\ell\dd\bigl(m\otimes\tilde{\mu}_e^{\otimes\N}\bigr)&=\int_{\set{\ell(w_0)> \ell}}\int_Xf(x,T_\ell(w_m)_m)\dd m(x)\dd \tilde{\mu}_e^{\otimes \N}((w_m)_m)\\
&=\int_X\int_{\set{\tau_e> \ell}}f(x,T^\ell\omega)\dd \Pro_e(\omega)\dd m(x).
\end{align*}
Noting that by the Markov property and the description \eqref{stat_dist} of $\pi$ we have 
\begin{align*}
\sum_{\ell=0}^\infty\int_{\set{\tau_e> \ell}}f(x,T^\ell\omega)\dd \Pro_e(\omega)=\E_e\sqbr*{\sum_{k=0}^{\tau_e-1}\E_{\omega_k}[f(x,\cdot)]}=\E_e[\tau_e]\E_\pi[f(x,\cdot)]
\end{align*}
for every $x\in X$, the remainder of the argument is the same as above. 
Indeed, together with dominated convergence this implies that the limit 
\begin{align*}
\frac{1}{\E_e[\tau_e]}\int_X\sum_{\ell=0}^{L-1}\int_{\set{\tau_e> \ell}}f(x,T^\ell\omega)\dd \Pro_e(\omega)\dd m(x)
\end{align*}
of \eqref{convergence_to_correct} now converges to $\int f\dd(m\otimes \Pro_\pi)$ as $L\to\infty$, and \eqref{negligible1} and \eqref{negligible2} still tend to $0$. 
\end{proof}
\subsection{Moment and Expansion Conditions}\label{subsec:markov_mom_exp}
We now express the notions of finite moments and expansion in the Markovian setting, in a way that will be convenient when combining the results of~\S\ref{sec:iid_walks} and~\S\ref{subsec:markov_equidist_bootstrap}. 

Let $\rho$ be a representation of $G$ on a finite-dimensional real vector space $V$. 
Recall that $N(g)=\max(\norm{\rho(g)},\norm{\rho(g)^{-1}})$, where $\norm{\cdot}$ is the operator norm associated to a fixed norm on $V$. 
\begin{definition}
A Markov chain on $E$ is said to have \emph{finite first moments} in $(V,\rho)$ if for every $e\in E$ 
\begin{align*}
\E_e[\log N(g_{\omega|_{\tau_e}})]<\infty,
\end{align*}
and to have \emph{finite exponential moments} in $(V,\rho)$ if for every $e\in E$ there exists $\delta>0$ such that 
\begin{align*}
\E_e[N(g_{\omega|_{\tau_e}})^\delta]<\infty.
\end{align*}
\end{definition}
%These definitions should be thought of as generalizations of positive and exponential recurrence taking into account the growth of the group elements associated to abstract states via the coding map $E\to G$. 
%Notice that this definition depends on the choice of coding map. 
%Moreover, it is clear that finite exponential moments imply a finite first moment. 
As usual, we suppress the representation from the notation when $(V,\rho)=(\g,\Ad)$. 
Note that the definition does not depend on the choice of norm on $V$.

In terms of renewal measures these conditions read as follows. 
\begin{lemma}
A recurrent irreducible Markov chain on $E$ has finite first \textup{(}resp.\ exponential\textup{)} moments in $V$ if and only if all renewal measures $\mu_e$ have the corresponding property. \qed
\end{lemma}
Let us mention a few simple examples in which the above moment conditions are satisfied. 
\begin{example}\label{ex:moments}
\leavevmode
\begin{enumerate}
\item If the state space $E$ is finite, then any irreducible Markov chain on $E$ has finite exponential moments in $(V,\rho)$. 
\item More generally, if the Markov chain on $E$ is irreducible and positive (resp.\ exponentially) recurrent and the coding map $E\to G$ takes values in a bounded subset of $G$, then the Markov chain has finite first (resp.\ exponential) moments in $(V,\rho)$. 
This conclusion stays valid when the coding map has sufficiently slow growth. 
\item Suppose the Markov chain on $E$ is positive recurrent and let $\pi$ be its stationary distribution. 
Denote by $c\colon E\to G$ the coding map. 
If $c_*\pi$ has a finite first moment in $(V,\rho)$, i.e.\ if 
\begin{align*}
\sum_{e'\in E}\log N(g_{e'})\pi(\set{e'})<\infty,
\end{align*}
then the Markov chain has finite first moments in $(V,\rho)$. 
\end{enumerate}
We omit the straightforward verifications. 
\end{example}

Next, we generalize the notion of uniform expansion from~\S\ref{sec:iid_walks}. 
\begin{definition}
Let $(Y_n)_n$ be a stochastic process with values in $\GL_d(\R)$ and $P\subset \Pro(\R^d)$ a closed subset invariant under the support of the distribution of $Y_n$ for all $n\in\N$. 
Then we call $(Y_n)_n$ \emph{uniformly expanding on $P$} if for all $\R v\in P$, almost surely, 
\begin{align*}
\liminf_{n\to\infty}\tfrac{1}{n}\log\norm{Y_n\dotsm Y_1v}>0.
\end{align*}
We call $(Y_n)_n$ \emph{uniformly expanding on Grassmannians} if $(Y_n^{\wedge k})_n$ is uniformly expanding on $\Pro(\pt^k\R^d)$ for all $1\le k\le d-1$. 
\end{definition}
%When the $Y_n$ are i.i.d.\ with distribution $\mu$, then this notion clearly coincides with the notions of uniform expansion of $\mu$ introduced in~\S\ref{sec:iid_walks}. 

To efficiently deal with our setting involving an abstract Markov chain on $E$, different starting distributions, and a coding map, it will be convenient to introduce the following more concise terminology. 
\begin{definition}
Let $\lambda$ be a starting distribution on $E$. 
Then we say that a Markov chain on $E$ is \emph{$\lambda$-expanding \textup{(}under the coding map $e\mapsto g_e$\textup{)}} if the stochastic process 
\begin{align*}
(Y_n)_n\colon(\Omega,\Pro_\lambda)\ni \omega\mapsto\bigl(\Ad(g_{\omega_{n-1}})\bigr)_n
\end{align*}
on $\GL(\g)$ is uniformly expanding on Grassmannians. 
When $\lambda=\delta_e$ for some $e\in E$ we also say that it is \emph{$e$-expanding}. 
\end{definition}
For brevity, we will usually omit the coding map from the notation when using these notions of expansion. 

Under a moment assumption as in Example~\ref{ex:moments}(iii), $e$-expansion can be phrased in terms of the renewal measure $\mu_e$. 
\begin{lemma}\label{lem:markov_expansion_equivalence}
Suppose that the Markov chain on $E$ is irreducible and recurrent and let $e\in E$. 
Denote by $\pi$ its stationary distribution and by $c\colon E\to G$ the coding map. 
\begin{enumerate}
\item If the Markov chain is $e$-expanding, then $\Ad_*\mu_e$ is uniformly expanding on Grassmannians. 
\item Suppose the Markov chain is additionally positive recurrent and $c_*\pi$ has a finite first moment in $\g$. %, i.e.\
%\begin{align*}
%\sum_{e'\in E}\log N(g_{e'})\pi(\set{e'})<\infty.
%\end{align*}
Then the Markov chain is $e$-expanding if and only if $\Ad_*\mu_e$ is uniformly expanding on Grassmannians. 
\end{enumerate}
\end{lemma}
\begin{proof}
Let $1\le k\le \dim(G)-1$. 
We $\Pro_e$-a.s.\ have $\tau_e^n/n\to\E_e[\tau_e]\in [1,\infty]$ as $n\to\infty$. 
By definition, $e$-expansion means that, $\Pro_e$-a.s., 
\begin{align}\label{liminf_pos1}
\liminf_{n\to\infty}\tfrac{1}{n}\log\norm{\Ad^{\wedge k}(g_{\omega|_n})v}>0.
\end{align}
From this it follows that $\Pro_e$-a.s.\ also 
\begin{align}\label{liminf_pos2}
\liminf_{n\to\infty}\tfrac{1}{n}\log\norm{\Ad^{\wedge k}(g_{\omega|_{\tau_e^n}})v}=\liminf_{n\to\infty}\tfrac{\tau_e^n}{n}\tfrac{1}{\tau_e^n}\log\norm{\Ad^{\wedge k}(g_{\omega|_{\tau_e^n}})v}>0.
\end{align}
This gives part (i). 
In the setting of (ii), we have $\E_e[\tau_e]\in[1,\infty)$, and the moment assumption allows applying Oseledets' theorem with the shift map on $(\Omega,\Pro_\pi)$ (we remark that Oseledets' theorem holds not only for i.i.d.\ processes, but more generally for stationary ones; see e.g.\ \cite[Theorem~1.6]{Ru}). 
We find that all the limit inferiors above are actually limits $\Pro_\pi$-, thus in particular $\Pro_e$-a.s., so that in this case \eqref{liminf_pos2} also implies \eqref{liminf_pos1}. 
\end{proof}
\subsection{Expanding Markov Chains}\label{subsec:exp_markov}
We now combine the bootstrapping results from~\S\ref{subsec:markov_stat_bootstrap} and~\S\ref{subsec:markov_equidist_bootstrap} with those of~\S\ref{sec:iid_walks} to prove our main Markovian measure classification and equidistribution results. 
These will imply Theorem~\ref{thm:intro2}. 
Recall that for $e\in E$, $G_e$ denotes the closed subgroup of $G$ generated by the support of the renewal measure $\mu_e$. 
\begin{theorem}\label{thm:exp_markov_stat}
Let $G$ be a real Lie group, $\Gamma$ a discrete subgroup of $G$, and $X$ the homogeneous space $G/\Gamma$. 
Suppose that the Markov chain on $E$ is irreducible and positive recurrent; denote by $\pi$ its stationary distribution. 
Suppose furthermore that the Markov chain is $\pi$-expanding and has finite first moments in $\g$. 
Let $\nu$ be an ergodic stationary probability measure for the action chain on $E\times X$ as in \eqref{stat_decomp}. 
Then either
\begin{enumerate}
\item for every $e\in E$ the measure $\nu_e$ is $G_e$-invariant and supported on a finite $G_e$-orbit, or 
\item $\Gamma$ is a lattice and all $\nu_e$ are the Haar measure $m_X$ on $X$. 
\end{enumerate}
Moreover, for every \adm{e}{e'}-admissible word $c\in E^*$, we have $(g_c)_*\nu_e=\nu_{e'}$. 
\end{theorem}
\begin{proof}
Note that by irreducibility of the Markov chain on $E$, $\pi$-expansion implies $e$-expansion for every $e\in E$. 
Thus, it follows by Lemma~\ref{lem:markov_stat_bootstrap} and Theorem~\ref{thm:exp_stat} that each $\nu_e$ is either supported on a finite $G_e$-orbit or is the Haar measure $m_X$ on $X$. 
Irreducibility of the Markov chain together with the last statement of Lemma~\ref{lem:markov_stat_bootstrap} imply that the same option applies to all $e\in E$. 

The last claim is clear in case (ii). 
In case (i), Lemma~\ref{lem:markov_stat_bootstrap} implies that $(g_c)_*\nu_e$ and $\nu_{e'}$ are of the same measure class. 
Being uniform measures on finite orbits, this forces $(g_c)_*\nu_e=\nu_{e'}$, as claimed. 
\end{proof}
\begin{theorem}\label{thm:exp_markov_equidist}
Let $G$ be a real Lie group with simple identity component such that the Zariski closure of $\Ad(G)$ is Zariski connected, $\Gamma$ a lattice in $G$, and $X=G/\Gamma$. 
Suppose that the Markov chain on $E$ is irreducible and positive recurrent and has finite exponential moments in $\g$. Denote by $\pi$ its stationary distribution and let $e\in E$. 
Assume that $G_e$ is not virtually contained in any conjugate of $\Gamma$, that $G_e$ acts transitively on the connected components of $X$, and that the Markov chain is $e$-expanding. 
Then, for every $x_0\in X$, $(g_{\omega|_n}x_0,T^n\omega)_n$ equidistributes towards $m_X\otimes \Pro_\pi$ for $\Pro_e$-a.e.\ $\omega\in\Omega$. 
\end{theorem}
\begin{proof}
Combine Lemma~\ref{lem:markov_expansion_equivalence}, Theorem~\ref{thm:exp_equidist}, and Proposition~\ref{prop:markov_equidist_bootstrap}. 
\end{proof}
\begin{proof}[Proof of Theorem~\textup{\ref{thm:intro2}}]
Let $\lambda$ denote the distribution of $Y_1$. 
By hypothesis, $(Y_m)_m$ is $\lambda$-expanding, hence $e$-expanding for every $e\in E$ with $\lambda(\set{e})>0$. 
Moreover, since $E$ is finite, the Markov chain on $E$ is positive recurrent and has finite exponential moments in $\g$. 
Now, in view of Lemma~\ref{lem:inf_orbits} below, the result follows by applying Theorem~\ref{thm:exp_markov_equidist} to each such $e\in E$. 
\end{proof}
\begin{lemma}\label{lem:inf_orbits}
Suppose that $E$ is finite and the Markov chain on $E$ is irreducible. 
If $x\in X$ and $e\in E$ are such that the random orbit $\set{g_{\omega|_n}x\for n\in\N}\subset X$ is $\Pro_e$-a.s.\ infinite, then the orbit $G_e^+x$ is infinite. 
\end{lemma}
\begin{proof}
Denote by $E_e^{\mathrm{adm}}$ the set of all admissible words starting with $e$ and consider the set 
\begin{align*}
\mathcal{O}=\set{g_wx\for w\in E_e^{\mathrm{adm}}}.
\end{align*}
By assumption it is infinite. 

Since the state space is finite, we can choose $k\in \N$ such that any state can be reached from everywhere in at most $k$ steps with positive probability. 
Then for every $w\in E_e^{\mathrm{adm}}$ there is an admissible word $c\in E^*$ of length at most $k-1$ such that $cw$ is \adm{e}{e}-admissible. 
It follows that $g_{cw}x\in G_e^+x$ and hence 
\begin{align*}
\mathcal{O}\subset \bigcup_{\substack{c\in E^*\text{ admissible}\\\ell(c)\le k-1}}g_c^{-1}G_e^+x,
\end{align*}
which forces $G_e^+x$ to be infinite as well. 
\end{proof}
\subsection{An Example}\label{subsec:upper_block_form}
To conclude this section, we are going to explain an example due to Simmons--Weiss~\cite{SW} that is used to relate Diophantine properties of fractals to random walks. 
We prove Proposition~\ref{prop:block_markov}, which can be considered a Markovian extension of \cite[Theorem~6.4]{SW}, and deduce Corollary~\ref{cor:intro4}. 

Let $G=\PGL_d(\R)$ and $\Gamma=\PGL_d(\Z)$. 
Given positive integers $M$ and $N$ with $M+N=d$, let $\R^{M\times N}$ be the space of $M\times N$-matrices with real entries and define 
\begin{align*}
a_t=\begin{pmatrix}\euler^{t/M}\mathbf{1}_M&\\&\euler^{-t/N}\mathbf{1}_N\end{pmatrix},\,u_{\boldsymbol{\alpha}}=\begin{pmatrix}\mathbf{1}_M&-\boldsymbol{\alpha}\\&\mathbf{1}_N\end{pmatrix},\text{ and }O_1\oplus O_2=\begin{pmatrix}O_1&\\&O_2\end{pmatrix}
\end{align*}
for $t\in\R$, $\boldsymbol{\alpha}\in\R^{M\times N}$ and $O_1\in\Orth_M(\R)$, $O_2\in\Orth_N(\R)$. 
We will denote the corresponding subgroups of $G$ by $A=\set{a_t\for t\in\R}$, $U=\set{u_{\boldsymbol{\alpha}}\for \boldsymbol{\alpha}\in\R^{M\times N}}$, $K=\set{O_1\oplus O_2\for O_1\in\Orth_M(\R),O_2\in\Orth_N(\R)}$, and set $P=AKU$. 
Note that $A$ and $K$ commute and normalize $U$; in particular, $P$ is a group. 
An element $g\in P$ can be uniquely written as a product of the form $a_tku_{\boldsymbol{\alpha}}$ and we denote the corresponding values of $t,k,\boldsymbol{\alpha}$ by $t(g),\boldsymbol{\alpha}(g)$ and $k(g)$, respectively. 
Finally, let $V^+=\Lie(U)$ be the Lie algebra of $U$. 
\begin{proposition}\label{prop:block_markov}
Suppose that $E$ is finite and let $\pi$ be the stationary distribution of an irreducible Markov chain on $E$. 
Suppose that the coding map $E\to G,\,e\mapsto g_e,$ takes values in $P$, that 
\begin{align}\label{exp_on_avg}
\sum_{e'\in E}t(g_{e'})\pi(\set{e'})>0,
\end{align}
and that for some $e_0\in E$ the Lie algebra of $H_{e_0}$ contains $V^+$. 
Then the assumptions of Theorem~\textup{\ref{thm:exp_markov_equidist}} are satisfied for every $e\in E$. 
\end{proposition}
\begin{proof}
Positive recurrence and finite exponential moments in $\g$ follow from finiteness of the state space. 
Below, we are going to show that all renewal measures $\mu_e$ are in \emph{$(M,N)$-upper block form} in the sense of \cite[Definition~6.3]{SW}. 
Then \cite[Theorem~6.4]{SW} (the proof of which does not use the assumption of compact support) implies that for every $e\in E$, $G_e$ is not virtually contained in any conjugate of $\Gamma$ and that Proposition~\ref{prop:SW_assumptions} can be applied to $\mu_e$, yielding $e$-expansion of the Markov chain. 

To show that $\mu_e$ is in $(M,N)$-upper block form for every $e\in E$, we have to argue that $\int_Gt(g)\dd\mu_e(g)>0$ and that the Lie algebra of $H_e$ contains $V^+$. 

Regarding positivity of the integral, we calculate, using that $t\colon P\to (\R,+)$ is a homomorphism and \eqref{stat_dist}, 
\begin{align*}
\int_Gt(g)\dd\mu_e(g)=\E_e[t(g_{\omega|_{\tau_e}})]&=\sum_{e'\in E}t(g_{e'})\E_e\sqbr*{\sum_{k=0}^{\tau_e-1}\mathbf{1}_{\omega_k=e'}}\\
&=\E_e[\tau_e]\sum_{e'\in E}t(g_{e'})\pi(\set{e'})>0.
\end{align*}
Finally, in view of the assumption on $H_{e_0}$, the inclusion $V^+\subset\Lie(H_e)$ follows from part (ii) of Lemma~\ref{lem:Z_closures_conjugate} and the fact that $U$ is normalized by $P$. 
\end{proof}
\begin{proof}[Proof of Corollary~\textup{\ref{cor:intro4}}]
By Proposition~\ref{prop:block_markov} and part (v) of Lemma~\ref{lem:Z_closures_conjugate} we need only verify that the Lie algebra of $H_\calS$ contains $V^+$. 
(Recall that part of the conclusion of Lemma~\ref{lem:Z_closures_conjugate} is that $H_\calS$ is in fact a group; here it is the real algebraic subgroup of $G$ generated by $g_0,\dots,g_r$.) 
The argument for this is the same as in the proof of \cite[Theorem~1.1]{SW}. 
Let us briefly reproduce it: 
For $0\le i\le r$, we have $g_i=u_i'a_ik_i$ with 
\begin{align*}
u_i'=\begin{pmatrix}\mathbf{1}_d&c_i^dy_i\\0&1\end{pmatrix},\,a_i=\begin{pmatrix}c_i\mathbf{1}_d&0\\0&c_i^{-d}\end{pmatrix},\text{ and }k_i=\begin{pmatrix}O_i&0\\0&1\end{pmatrix}.
\end{align*}
Then, for $n\in\N$, we can write 
\begin{align*}
H_\calS\ni g_0^{-n}g_ig_0^n=(k_0^{-n}a_0^{-n}u_i'a_0^nk_0^n)a_i(k_0^{-n}k_ik_0^n).
\end{align*}
Noting that for $n\to\infty$ we have $a_0^{-n}u_i'a_0^n\to\mathbf{1}_{d+1}$ and passing to a subsequence along which $k_0^{n_j}\to\mathbf{1}_{d+1}$ as $j\to\infty$, it follows that $a_ik_i\in H_\calS$, so that also $u_i'\in H_\calS$. 
Thus, we see that $M_j\coloneqq k_0^{-n_j}a_0^{-n_j}u_i'a_0^{n_j}k_0^{n_j}\in H_\calS\cap U$ for all $j$. 
Since $M_j\to \mathbf{1}_{d+1}$ as $j\to\infty$, this implies that 
\begin{align*}
\Lie(H_\calS)\ni \log(M_j)=M_j-\mathbf{1}_{d+1}
\end{align*}
for $j$ large enough. 
As a computation shows, the right-hand side above converges in direction towards $(\begin{smallmatrix}\mathbf{0}_d&y_i\\0&0\end{smallmatrix})$. 
Since the $y_i$ span $\R^d$ by assumption, we conclude that indeed $V^+\subset \Lie(H_\calS)$. 
\end{proof}
\section{Diophantine Approximation on Fractals}\label{sec:fractals}
As observed by Simmons--Weiss, equidistribution results as in~\S\ref{sec:iid_walks} can be used to obtain statements about Diophantine approximation on fractals obtained as limit sets of similarity IFS. 
In this final section, using the analogous results for Markov random walks from~\S\ref{sec:markov_walks}, we deal with limit sets of \emph{graph directed similarity IFS}. 

The first three subsections are of preparatory nature. 
We recall basic terminology and results on graph directed IFS (\S\ref{subsec:gd_IFS}), and make the connection between similarities, the homogeneous dynamics setting and Diophantine approximation (\S\ref{subsec:dioph},~\S\ref{subsec:alg_sim}). 
Our main Diophantine approximation results, which imply Theorems~\ref{thm:intro3} and~\ref{thm:intro4}, will be stated and proved in~\S\ref{subsec:approx_result}. 
\subsection{Graph Directed IFS}\label{subsec:gd_IFS}
Recall that a \emph{directed multigraph} is a tuple $(V,E,i,t)$ consisting of non-empty sets $V,E$ of \emph{vertices} and \emph{edges}, respectively, and functions $i,t\colon E\to V$ associating to an edge $e\in E$ the \emph{initial vertex} $i(e)\in V$ and the \emph{terminal vertex} $t(e)\in V$. 
The multigraph is \emph{finite} if both sets $V$ and $E$ are. 
A non-empty word $w=e_0\ldots e_{n-1}\in E^*$ or sequence $\omega=(e_m)_m\in E^\N$ is called a \emph{\textup{(}finite resp.\ infinite\textup{)} path} if $t(e_{j-1})=i(e_j)$ for all $j$. 
Denote the set of infinite paths by $E^\infty$. 
We extend the initial vertex function $i$ to paths by $i((e_m)_m)=i(e_0\ldots e_{n-1})=i(e_0)$, and the terminal vertex function $t$ to finite paths by $t(e_0\ldots e_{n-1})=t(e_{n-1})$. 
We call the multigraph \emph{connected} if for every pair of vertices $u,v\in V$ there exists a finite path from $u$ to $v$ (i.e.\ a path $w$ with $i(w)=u$ and $t(w)=v$). 
Finally, we call a Markov chain on $E$ (or an associated Markov measure on $E^\N$) \emph{adapted} if the transition probabilities $(p_{e',e})_{e,e'\in E}$ satisfy $p_{e',e}>0\iff t(e)=i(e')$ for $e,e'\in E$. 
Observe that if the multigraph is connected, any adapted shift-invariant Markov measure on $E^\infty$ is ergodic. 
\begin{remark}
When $E$ is finite, the space $E^\infty\subset E^\N$ of infinite paths is the subshift of finite type defined by the edge-incidence relation given by the multigraph. 
The notation is intentionally the same as for admissible sequences in~\S\ref{sec:markov_walks}, since these notions coincide for adapted Markov chains on $E$, to which we will from now on restrict our attention. 
\end{remark}
Recall that a \emph{similarity} of $\R^d$ is a map $\phi\colon\R^d\to\R^d$ of the form $\phi(x)=rO(x)+b$ for some $r>0$, $O\in\Orth_d(\R)$ and $b\in\R^d$. 
The number $r=\norm{\phi'}$ is the \emph{similarity ratio} of $\phi$. 
If $r<1$, $\phi$ is said to be \emph{contracting}. 
\begin{definition}\label{def:gd_IFS}
Let $(V,E,i,t)$ be a finite connected directed multigraph and suppose that for every $e\in E$ we are given a similarity $\phi_e\colon\R^d\to\R^d$. 
Then the tuple $(V,E,i,t,(\phi_e)_e)$ is called a \emph{graph directed similarity IFS}. 
\end{definition}
Note that ordinary similarity IFS represent the special case of graph directed similarity IFS with a single vertex. 
We also emphasize that finiteness and connectedness of the directed multigraph are part of our definition of graph directed similarity IFS. 

It is customary to think of one copy of $\R^d$ being attached to each vertex, and the map $\phi_e$ going from the copy at $t(e)$ to the one at $i(e)$. 
This viewpoint is consistent with the formula $\phi_w=\phi_{e_0}\dotsm \phi_{e_{n-1}}$ for words $w=e_0\ldots e_{n-1}$, which, incidentally, also explains why we now use a left-to-right indexing convention. 

We need to introduce some more terminology. 
A graph directed similarity IFS is said to be 
\begin{itemize}
\item \emph{contracting} if $\sup_{e\in E}\norm{\phi'_e}<1$, 
\item to satisfy the \emph{open set condition} if there exists a collection $(U_v)_{v\in V}$ of non-empty open subsets of $\R^d$ with $\phi_e(U_{t(e)})\subset U_{i(e)}$ for every $e\in E$ and $\phi_e(U_{t(e)})\cap\phi_{e'}(U_{t(e')})=\emptyset$ for any distinct edges $e,e'\in E$ with $i(e)=i(e')$, and 
\item to be \emph{irreducible} if there does not exist a collection $(\mathcal{L}_v)_{v\in V}$ of proper affine subspaces of $\R^d$ with $\phi_e(\mathcal{L}_{t(e)})=\mathcal{L}_{i(e)}$ for every $e\in E$. 
\end{itemize}

Given a contracting graph directed similarity IFS, one proves in complete analogy to the classical case that there is a unique collection $(K_v)_{v\in V}$ of non-empty compact subsets of $\R^d$ such that 
\begin{align*}
K_v=\bigcup_{i(e)=v}\phi_e(K_{t(e)})
\end{align*}
for every $v\in V$ (see \cite{MW}). 
The union $K=\bigcup_{v\in V}K_v$ is called the \emph{attractor} of the graph directed IFS. 
It can alternatively be obtained as the image of $E^\infty$ under the \emph{natural projection} 
\begin{align*}
\Pi\colon E^\infty\to\R^d,\,\omega\mapsto \lim_{n\to\infty}\phi_{\omega_0}\dotsm\phi_{\omega_{n-1}}(x),
\end{align*}
which is continuous and independent of the choice of $x\in \R^d$. 
Observe that the attractors $K$ arising in this way are precisely what we called sofic similarity fractals in~\S\ref{subsec:fractal_intro}. 
Indeed, setting $\Phi=\set{\phi_e\for e\in E}$, the image of $E^\infty$ under the map $E^\infty\to\Phi^\N,\,\omega\mapsto(\phi_{\omega_m})_m,$ is a sofic subshift of $\Phi^\N$. 

Generalizing a classical result of Hutchinson~\cite{Hut}, Wang~\cite{Wa} identified the Hausdorff measure on attractors of graph directed similarity IFS satisfying the open set condition. 
\begin{theorem}[Wang~\cite{Wa}]\label{thm:gd_IFS}
Let $(V,E,i,t,(\phi_e)_e)$ be a contracting graph directed similarity IFS satisfying the open set condition. 
Let $K$ be the associated attractor, $s\ge 0$ its Hausdorff dimension, $\Pi$ the natural projection, and denote $s$-dimensional Hausdorff measure by $\calH^s$. 
Then $\calH^s|_K$ is proportional to $\Pi_*\Pro$ for some adapted shift-invariant Markov probability measure $\Pro$ on $E^\infty$. 
\end{theorem}
\subsection{Diophantine Approximation and Dani Correspondence}\label{subsec:dioph}
Recall that a matrix $\boldsymbol{\alpha}\in\R^{M\times N}$ is said to be 
\begin{itemize}
\item \emph{badly approximable} if there exists $c>0$ such that for all $\mathbf{q}\in\Z^N\setminus\set{0}$ and $\mathbf{p}\in\Z^M$ we have $\norm{\boldsymbol{\alpha}\mathbf{q}-\mathbf{p}}\ge c\norm{\mathbf{q}}^{-N/M}$, 
\item \emph{well approximable} if it is not badly approximable, and 
\item \emph{Dirichlet improvable} if there exists $0<\lambda<1$ such that  for all sufficiently large $Q$ there exist $\mathbf{q}\in\Z^N\setminus\set{0}$ with $\norm{\mathbf{q}}_\infty\le \lambda Q$ and $\mathbf{p}\in\Z^M$ with $\norm{\boldsymbol{\alpha}\mathbf{q}-\mathbf{p}}_\infty\le \lambda Q^{-N/M}$. 
\end{itemize}
In the above, $\norm{\cdot}_\infty$ denotes the supremum norm on $\R^{M\times N}$ and $\norm{\cdot}$ an arbitrary norm. 
A general survey of Diophantine approximation can be found in~\cite{BRV}. 
For a more specific overview pertaining to the topic at hand we refer to~\cite[\S7]{SW}. 

The Dani correspondence principle asserts that the Diophantine properties of $\boldsymbol{\alpha}$ are encoded in the behavior of the orbit $\set{a_tu_{\boldsymbol{\alpha}}\SL_d(\Z)}_{t\ge 0}$ in $X=\SL_d(\R)/\SL_d(\Z)$ (using the notation from~\S\ref{subsec:upper_block_form}). 
To see this, it is useful to think of $X$ as the space $X_d$ of unimodular lattices in $\R^d$, via the identification 
\begin{align*}
X\ni g\SL_d(\Z)\longleftrightarrow g\Z^d\in X_d.
\end{align*}
The Mahler compactness criterion then says that a subset $A\subset X$ is relatively compact if and only if it is contained in one of the sets 
\begin{align*}
K_\epsilon=\set{x\in X\for\forall v\in x\setminus\set{0}\colon \norm{v}_\infty\ge \epsilon}
\end{align*}
for $0<\epsilon<1$. 
Note that these sets themselves are compact, exhaust $X$, and satisfy $K_{\epsilon_1}^\circ\supset K_{\epsilon_2}$ for $0<\epsilon_1<\epsilon_2$. 
\begin{theorem}[Dani correspondence]
The matrix $\boldsymbol{\alpha}\in\R^{M\times N}$ is 
\begin{enumerate}
\item badly approximable if and only if the trajectory $\set{a_tu_{\boldsymbol{\alpha}}\SL_d(\Z)}_{t\ge 0}$ is relatively compact, i.e.\ contained in $K_\epsilon$ for some $0<\epsilon<1$, 
\item Dirichlet improvable if and only if for some $0<\lambda<1$ the trajectory $\set{a_tu_{\boldsymbol{\alpha}}\SL_d(\Z)}_{t\ge 0}$ eventually leaves $K_\lambda$, i.e.\ if there exists $T\ge 0$ such that $\set{a_tu_{\boldsymbol{\alpha}}\SL_d(\Z)}_{t\ge T}$ does not intersect $K_\lambda$. 
\end{enumerate}
\end{theorem}
For the proofs, we refer to Dani~\cite[Theorem~2.20]{Da} and Kleinbock--Weiss~\cite[Proposition~2.1]{KW}. 
\begin{corollary}\label{cor:dense_orbit}
If $\set{a_tu_{\boldsymbol{\alpha}}\SL_d(\Z)}_{t\ge 0}$ is dense in $X$, then $\boldsymbol{\alpha}$ is well approximable and not Dirichlet improvable. 
\end{corollary}
%Incidentally, Davenport and Schmidt (who originally introduced the concept of Dirichlet improvability) proved in~\cite{DS} that badly approximable matrices are automatically Dirichlet improvable. 
%We shall not make use of this fact. 
In fact, the random walk approach yields the following stronger property. 
\begin{definition}
A matrix $\boldsymbol{\alpha}\in\R^{M\times N}$ is said to be \emph{of generic type} if the orbit $\set{a_tu_{\boldsymbol{\alpha}}\SL_d(\Z)}_{t\ge 0}$ is equidistributed in $X$ with respect to the Haar measure $m_X$. 
\end{definition}
\subsection{Algebraic Similarities as Group Elements}\label{subsec:alg_sim}
Next, following~\cite[\S10]{SW}, we interpret a class of similarities of $\R^{M\times N}$ as elements of $\PGL_d(\R)$. 

Recall the subgroups $A,K,U$ and $P=AKU$ of $\PGL_d(\R)$ defined in~\S\ref{subsec:upper_block_form}. 
The group $P$ acts by left multiplication on the space $P/AK$, which is topologically identified with $U\cong \R^{M\times N}$ via 
\begin{align*}
\R^{M\times N}\ni \boldsymbol{\beta}\longleftrightarrow u_{-\boldsymbol{\beta}}AK\in P/AK.
\end{align*}
The obtained action of $P$ on $\R^{M\times N}$ is faithful and is described as follows: For $\boldsymbol{\beta}\in\R^{M\times N}$ we have
\begin{align*}
a_t\acts \boldsymbol{\beta}&=\euler^{t(1/M+1/N)}\boldsymbol{\beta},\\
k\acts \boldsymbol{\beta}&=O_1\boldsymbol{\beta}O_2^{-1},\\
u_{\boldsymbol{\alpha}}\acts \boldsymbol{\beta}&=\boldsymbol{\beta}-\boldsymbol{\alpha},
\end{align*}
for $a_t\in A$, $k=O_1\oplus O_2\in K$ and $u_{\boldsymbol{\alpha}}\in U$. 
Thus, $P$ can be identified with the group of \emph{algebraic similarities} of $\R^{M\times N}$, i.e.\ similarities of the form $\boldsymbol{\beta}\mapsto rO_1\boldsymbol{\beta}O_2+\boldsymbol{\alpha}$ for some $r>0$, $O_1\in \Orth_M(\R)$, $O_2\in \Orth_N(\R)$ and $\boldsymbol{\alpha}\in\R^{M\times N}$. 
Note that when $M=1$ or $N=1$, all similarities of $\R^{M\times N}$ are algebraic. 
%\begin{remark}\label{rmk:SW_sign_problem}
%Simmons--Weiss use the identification $\boldsymbol{\beta}\leftrightarrow u_{\boldsymbol{\beta}}$ and the corresponding formula $u_{\boldsymbol{\alpha}}\acts \boldsymbol{\beta}=\boldsymbol{\beta}+\boldsymbol{\alpha}$. 
%However, with this convention the signs in the proofs of \cite[Lemma~11.3]{SW} and \cite[Theorem~8.11]{SW} are incorrect. 
%Changing the identification to the one above fixes this problem. 
%\end{remark}
\subsection{The Approximation Result}\label{subsec:approx_result}
We are now ready to formulate and prove the graph directed version of \cite[Theorem~8.1]{SW}. 
\begin{theorem}\label{thm:gd_approx}
Let $(V,E,i,t,(\phi_e)_e)$ be a contracting irreducible graph directed IFS of algebraic similarities of $\R^{M\times N}$ satisfying the open set condition. 
Let $K$ denote its attractor and $s\ge 0$ its Hausdorff dimension. 
Then almost every point on $K$ with respect to $s$-dimensional Hausdorff measure is of generic type, so in particular, well approximable and not Dirichlet improvable. 
\end{theorem}
\begin{proof}[Proof of Theorem~\textup{\ref{thm:intro4}}]
As already remarked, in the case $N=1$ all similarities are algebraic. 
Now the result follows by an application of Theorem~\ref{thm:gd_approx}. 
\end{proof}

By virtue of Wang's Theorem~\ref{thm:gd_IFS}, Theorem~\ref{thm:gd_approx} above is a consequence of the following result. 
\begin{theorem}\label{thm:gd_approx_general}
Let $(V,E,i,t,(\phi_e)_e)$ be an irreducible graph directed similarity IFS on $\R^{M\times N}$ consisting of algebraic similarities, and $\Pro$ an adapted shift-invariant Markov measure on $E^\infty$ for which the IFS is \emph{contracting on average}, in the sense that
\begin{align*}
\sum_{e\in E}\log\norm{\phi_e'}\pi(\set{e})<0,
\end{align*}
where $\pi$ denotes the projection of $\Pro$ to the first coordinate. 
Then the natural projection $\Pi\colon E^\infty\to \R^{M\times N}$ is well-defined $\Pro$-almost everywhere and almost every point with respect to $\Pi_*\Pro$ is of generic type. 
\end{theorem}
\begin{proof}
Note that the natural projection is well-defined at $\omega\in E^\infty$ whenever the contraction ratios $\norm{\phi_{\omega_0\ldots\omega_{n-1}}'}$ decay exponentially. 
Recalling that adapted shift-invariant Markov measures are ergodic, it follows from the Birkhoff ergodic theorem and the contraction-on-average assumption that this is the case $\Pro$-a.s. 
What we need to show is that the orbit $\set{a_tu_{\Pi(\omega)}\SL_d(\Z)}_{t\ge 0}$ is equidistributed with respect to the Haar measure $m_X$ on $X=\SL_d(\R)/\SL_d(\Z)=\PGL_d(\R)/\PGL_d(\Z)$ for $\Pro$-a.e.\ $\omega\in E^\infty$. 

To see this, we follow Simmons--Weiss' strategy in the proof of \cite[Theorem~8.11]{SW} and connect the above orbit with certain random walk trajectories. 
First note that $\Pro$ defines an irreducible finite-state Markov chain on the set $E$ of edges. 
Using the construction in~\S\ref{subsec:alg_sim}, we can view the algebraic similarities $\phi_e$ as elements of $P\subset G=\PGL_d(\R)$. 
Defining the coding map
\begin{align*}
E\ni e\mapsto g_e\coloneqq\phi_e^{-1}\in P,
\end{align*}
we are then in the setting of~\S\ref{sec:markov_walks}. We claim that (after a conjugation) the assumptions of Proposition~\ref{prop:block_markov} are satisfied. 
Indeed, validity of \eqref{exp_on_avg} follows from the contraction-on-average assumption on the $\phi_e$ (notice the inverse in the definition of the $g_e$), and the assumption on the Lie algebra of $H_{e_0}$ for some $e_0\in E$ is satisfied after conjugating the coding map by an element of $P$ so that $H_{e_0}$ contains an element $h_0\in AK$ with $t(h_0)>0$, as the corresponding argument in~\cite[\S10.1]{SW} shows. 
One just needs to observe that the irreducibility assumption on the graph directed IFS forces the IFS consisting of the atoms of the renewal measure $\mu_{e_0}$ to be irreducible. 
(An invariant affine subspace $\mathcal{L}$ for the support of $\mu_{e_0}$ gives rise to an invariant collection of subspaces $(\mathcal{L}_v)_v$ in the graph directed sense by choosing for each vertex $v$ a path $w_v$ from $i(e_0)$ to $v$ starting with $e_0$ and setting $\mathcal{L}_v=\phi_{w_v}^{-1}(\mathcal{L})$.) 
We conclude that Theorem~\ref{thm:exp_markov_equidist} can be applied for every $e\in E$. 
Writing $\Pro$ as convex combination of the measures $\Pro_e$ as in \eqref{arb_markov}, we thus obtain $\Pro$-a.s.\ equidistribution of $(g_{\omega|_n}\SL_d(\Z))_n$ towards $m_X$. 

We shall use this to argue that the sequence 
\begin{align}\label{intermediate_goal}
(x_n,\omega_n)_n,\quad\text{with }x_n=k(g_{\omega|_n})^{-1}u_{\Pi(T^n\omega)}g_{\omega|_n}\SL_d(\Z),
\end{align}
equidistributes towards $m_X\otimes \pi$ for $\Pro$-a.e.\ $\omega\in E^\infty$, where $k(\cdot)$ denotes the $K$-component of an element of $P=AKU$. 
To this end, we consider the Markov random walk on $X\times K$ given by the coding map $E\ni e\mapsto (g_e,k(g_e))\in G\times K$ and the associated action chain trajectories 
\begin{align*}
(y_n)_n=\bigl(\omega_n,g_{\omega|_n}\SL_d(\Z),k(g_{\omega|_n})\bigr)_n
\end{align*}
in $E\times X\times K$. 
Since $\Pro$-a.s.\ the random walk trajectory $(g_{\omega|_n}\SL_d(\Z))_n$ equidistributes towards $m_X$, no escape of mass can occur for the sequence $\frac1n\sum_{k=0}^{n-1}\delta_{y_k}$ of empirical measures. 
The Breiman law of large numbers (see \cite[Corollary~3.3]{BQ3})
thus implies that $\Pro$-a.s.\ every weak* limit $\nu$ of this sequence of empirical measures is a probability measure on $E\times X\times K$ that is stationary for the action chain. 
By Lemma~\ref{lem:markov_stat_bootstrap}, $\nu$ decomposes as
\begin{align*}
\nu=\sum_{e\in E}\pi(\set{e})\delta_e\otimes \nu_e
\end{align*}
for $\mu_e$-stationary probability measures $\nu_e$ on $X\times K$. 
Using equidistribution of $(g_{\omega|_n}\SL_d(\Z))_n$ once more, we see that the $\nu_e$ project to $m_X$ in the first coordinate. 
Moreover, for every $e\in E$ the closed subgroup $G_e$ generated by the support of the renewal measure $\mu_e$ is non-compact and therefore acts mixingly on $X$ by the Howe--Moore theorem. 
Thus, its action on $X\times K_e$ is ergodic, where $K_e$ denotes the compact group $\overline{k(G_e)}$. 
These observations put us in a position to apply \cite[Proposition~5.3]{SW}. 
The conclusion is that $\nu_e=m_X\otimes m_{K_e}$, where $m_{K_e}$ is the Haar measure on $K_e$. 
Hence the limit $\nu$ is unique, so that $(y_n)_n$ equidistributes $\Pro$-a.s.\ towards 
\begin{align*}
\sum_{e\in E}\pi(\set{e})\delta_e\otimes m_X\otimes m_{K_e}.
\end{align*}
Now part (ii) of Lemma~\ref{lem:equidist_reduction} implies that 
\begin{align*}
(y_n,T^n\omega)_n
\end{align*}
equidistributes towards the probability measure 
\begin{align*}
\sum_{e\in E}\pi(\set{e})\delta_e\otimes m_X\otimes m_{K_e}\otimes \Pro_e
\end{align*}
on $E\times X\times K\times E^\infty$ for $\Pro$-a.e.\ $\omega\in E^\infty$. 

The natural projection $\Pi$ is not necessarily continuous in the contracting-on-average case. 
However, a standard argument involving Lusin's theorem still shows that the equidistribution of $(y_n,T^n\omega)_n$ established above entails $\Pro$-a.s.\ equidistribution of 
\begin{align*}
(y_n,\Pi(T^n\omega))_n
\end{align*}
towards 
\begin{align*}
\sum_{e\in E}\pi(\set{e})\delta_e\otimes m_X\otimes m_{K_e}\otimes \Pi_*\Pro_e
\end{align*}
(cf.\ the proof of \cite[Proposition~5.2]{SW}). 
Applying the continuous map 
\begin{align*}
F\colon E\times X\times K\times \R^{M\times N}&\to X\times E\\(e,x,k,\boldsymbol{\alpha})&\mapsto(k^{-1}u_{\boldsymbol{\alpha}}x,e) 
\end{align*}
we finally obtain equidistribution of \eqref{intermediate_goal} towards 
\begin{align*}
F_*\br[\bigg]{\sum_{e\in E}\pi(\set{e})\delta_e\otimes m_X\otimes m_{K_e}\otimes \Pi_*\Pro_e}=m_X\otimes \pi.
\end{align*}

Having established the necessary equidistribution for random walk trajectories, the final ingredient needed to finish the proof is the connection to the geodesic flow trajectory of $u_{\Pi(\omega)}\SL_d(\Z)$. 
It comes from the relationship 
\begin{align}\label{magic_formula}
x_n=a_{t_n}u_{\Pi(\omega)}\SL_d(\Z),
\end{align}
where $t_n=t(g_{\omega|_n})$. 
To verify this formula, one first notes that the $AK$-components of both sides agree. 
To see that the $U$-components do as well, one applies the inverses of $g_{\omega|_n}$ and $u_{\Pi(T^n\omega)}^{-1}k(g_{\omega|_n})a_{t_n}u_{\Pi(\omega)}$ interpreted as algebraic similarities to the matrix $\Pi(T^n\omega)$ and observes that the result is $\Pi(\omega)$ in both cases. 
%Here Remark~\ref{rmk:SW_sign_problem} is pertinent. 

Given a bounded continuous function $f$ on $X$, it now remains to apply  equidistribution of $(x_n,\omega_n)_n$ towards $m_X\otimes \pi$ to the function $f'$ on $X\times E$ defined by $f'(x,e)=\int_0^{t(g_e)}f(a_tx)\dd t$. 
As in the proof of \cite[Theorem~8.11]{SW}, in view of \eqref{magic_formula} this yields 
\begin{align*}
\lim_{n\to\infty}\frac{1}{t_n}\int_0^{t_n}f(a_tu_{\Pi(\omega)}\SL_d(\Z))\dd t=\frac{\int_{X\times E}f'\dd(m_X\otimes\pi)}{\int_Et(g_e)\dd\pi(e)}=\int_Xf\dd m_X.
\end{align*}
Using that the sequence $(t_n)_n$ has bounded gaps, this proves equidistribution of $\set{a_tu_{\Pi(\omega)}\SL_d(\Z)}_{t\ge 0}$ with respect to $m_X$. \qedhere
\end{proof}
\begin{proof}[Proof of Theorem~\textup{\ref{thm:intro3}}]
Consider a directed multigraph with a single vertex $v_0$ and edge set $E=\Phi$ (with $t(\phi)=i(\phi)=v_0$ for all $\phi\in\Phi$). 
Since the Markov measure $\Pro$ has full support, it defines an irreducible Markov chain on $\Phi$. 
Let $\pi$ be its stationary distribution and $\Pro_\pi$ the associated Markov measure. 
Then Theorem~\ref{thm:gd_approx_general} can be applied to $\Pro_\pi$ and yields the desired conclusion for $\Pi_*\Pro_\pi$-a.e.\ point on $K$. 
Noting that $\pi(\set{\phi})>0$ for all $\phi\in\Phi$ by irreducibility and using \eqref{arb_markov} once for $\pi$ and once for the projection of $\Pro$ to the first coordinate, we deduce that the conclusion holds $\Pi_*\Pro_\phi$-a.s.\ for every $\phi\in\Phi$, and thus also $\Pi_*\Pro$-a.s. 
\end{proof}

\bibliographystyle{plain}
\bibliography{refs}

\end{document}